\documentclass[11pt,leqno]{amsart}
\usepackage{amscd}
\usepackage{amssymb}
\usepackage{amsfonts}
\usepackage[centertags]{amsmath}
\usepackage{latexsym}
\usepackage{verbatim}
\usepackage{amsthm}
\usepackage{color}
\usepackage{pgfplots}
\usepackage{tikz}
\usetikzlibrary{intersections}
\usepackage{amsfonts}
\usepackage{amssymb}
\usepackage{amsmath}
\usepackage{amsthm}
\usepackage[english]{babel}
\usepackage{enumerate}
\usepackage{graphicx}
\usepackage{color}
\usepackage{hyperref}
\renewcommand{\H}{\mathbb{H}}

\newcommand{\NN}{\mathbb{N}}
\newcommand{\RR}{\mathbb{R}}

\newcommand{\pa}{\partial}

\newcommand{\al}{\alpha}

\newcommand{\ga}{\gamma}

\newcommand{\de}{\delta}

\newcommand{\ep}{\varepsilon}

\newtheorem{theorem}{Theorem}[section]
\newtheorem{corollary}[theorem]{Corollary}
\newtheorem{lemma}[theorem]{Lemma}
\newtheorem{proposition}[theorem]{Proposition}

\newtheorem{remark}[theorem]{Remark}

\theoremstyle{remark}

\DeclareMathOperator{\arccosh}{\mathrm{arccosh}}

\DeclareMathOperator{\osc}{\mathrm{osc}}
\DeclareMathOperator{\diam}{\mathrm{diam}}

\newcommand{\oscH}{\osc(H)}

\title[Almost CMC hypersurfaces in the hyperbolic space]{Quantitative Stability for hypersurfaces with almost constant mean curvature  in the Hyperbolic Space}

\author{Giulio Ciraolo, Luigi Vezzoni}

\date{\today}
\address{Dipartimento di Matematica e Informatica, Universit\`a di Palermo, Via Archirafi 34, 90123 Palermo, Italy} \email{giulio.ciraolo@unipa.it}

\address{Dipartimento di Matematica G. Peano, Universit\`a di Torino, Via Carlo Alberto 10, 10123 Torino, Italy.} \email{luigi.vezzoni@unito.it}


\begin{document}

\maketitle

\begin{abstract}
We provide sharp stability estimates for the Alexandrov Soap Bubble Theorem in the hyperbolic space. The closeness to a single sphere is quantified in terms of the  dimension, the measure of the hypersurface and the radius of the touching ball condition.  As consequence we obtain a new pinching result for hypersurfaces in the hyperbolic space.   

Our approach is based on the method of moving planes. In this context we carefully review the method and we provide the first quantitative study in the hyperbolic space.
\end{abstract}

\tableofcontents

\section{Introduction}
In this paper we study  compact embedded hypersurfaces in the hyperbolic space in relation to the mean curvature. The subject has been largely studied in literature (see e.g. \cite{Br,DCL, desilva,Gr,GuanSpruck1,GuanSpruck,GuanSpruckSzapiel,GuanSpruckXiao,Hs,Ko,Levitt,lopez1,lopez2,Mo,nelli,pacard,rosenberg,Yau} and the references therein).

Our starting point is the celebrated Alexandrov's theorem in the hyperbolic context:

\medskip

\noindent {\bf Alexandrov's theorem.} {\em A connected closed $C^2$-regular hypersurface $S$ embedded in the hyperbolic space has constant mean curvature if and only if it is a sphere.}

\medskip
The theorem was proved by Alexandrov in \cite{Al0} by using the method of moving planes and extends to the Euclidean space and the hemisphere \cite{Al0,Al1,Al2}. The method uses maximum principles and consists in proving that the surface is symmetric in any direction. Then the assertion follows by the following characterization of the sphere: a compact embedded hypersurface $S$ in the hyperbolic space with center of mass $\mathcal O$ is a sphere if and only if for every direction $\omega$ there exists a hyperbolic hyperplane $\pi_\omega$ of symmetry of $S$ orthogonal to $\omega$ at $\mathcal O$ (see lemma 2.2). 

In this paper we study the method of moving planes in the hyperbolic space from a quantitative point of view and we obtain sharp stability estimates for Alexandrov's theorem. We consider a $C^2$-regular, connected, closed hypersurface $S$ embedded in the hyperbolic space. Since $S$ is closed and embedded, there exists a bounded domain $\Omega$ such that $S=\partial \Omega$. We say that $S$ (or equivalently $\Omega$) satisfies a uniform touching ball condition of radius $\rho$ if for any point $p \in S $ there exist two balls ${\sf B}_\rho^-$ and ${\sf B}_\rho^+$ of radius $\rho$, with ${\sf B}_\rho^-$ contained $\Omega$ and ${\sf B}_\rho^+$ outside $\Omega$, which are tangent to $S$ at $p$. Our main result is the following.

\begin{theorem}\label{main}
Let $S$ be a $C^2$-regular, connected, closed hypersurface embedded in the $n$-dimensional hyperbolic space satisfying a uniform touching ball condition of radius $\rho$. There exist constants $\ep,\, C>0$ such that if the mean curvature $H$ of $S$ satisfies
\begin{equation}\label{H quasi const}
\oscH \leq \ep,
\end{equation}
then there are two concentric balls ${\sf B}_{r}$ and ${\sf B}_{R}$ such that
\begin{equation}\label{Bri Om Bre}
S \subset \overline{\sf B}_{R} \setminus {\sf B}_{r},
\end{equation}
and
\begin{equation}\label{stability radii}
R-r \leq C \oscH.
\end{equation}
The constants $\ep$ and $C$ depend only on $n$ and upper bounds on $\rho^{-1}$ and on the area of $S$.
\end{theorem}

In theorem \ref{main}, $\oscH$ is the oscillation of H, i.e. $\oscH:=\max_M H-\min_M H$. Note that the assumption $\oscH \leq \ep$ is equivalent to require that $H$ is close to a constant in $C^0$-norm.    
We remark that the quantitative bound in \eqref{stability radii} is sharp in the sense that no function of $\oscH$ converging to zero more than linearly can appear on the right hand side of \eqref{stability radii}, as can be seen by explicit calculations considering a small perturbation of the sphere. We prefer to state theorem \ref{main} by assuming that $S$ is connected, however the theorem still holds if we just assume that $\Omega$ is connected (and the proof remains the same).

Theorem \ref{main} has some remarkable consequence that we give in the following corollary.

\begin{corollary}\label{main2}
Let $\rho_0,A_0>0$ and $n \in \NN$ be fixed. There exists $\ep>0$, depending on $n$, $\rho_0$ and $A_0$, such that if $S$ is a connected closed $C^2$ hypersurface embedded in the hyperbolic space having area bounded by $A_0$, satisfying a touching ball condition of radius $\rho \geq \rho_0$, and whose mean curvature $H$ satisfies
\begin{equation*}
\oscH <\ep\,,
\end{equation*}
then $S$ is diffeomorphic to  a sphere. 

Moreover $S$ is $C^{1,\alpha}$-close to a sphere, i.e. there exists a $C^{1,\alpha}$-map $\Psi\colon  \partial {\sf B}_{r}\to \RR$ such that
$$
F(x)=\exp_x(\Psi(x)N_x)
$$
defines a $C^{1,\alpha}$-diffeomorphism $F\colon \partial {\sf B}_{r}\to S$ and 
\begin{equation} \label{Lipschitz_bound}
\|\Psi\|_{C^{1,\alpha}(\partial B_{r})} \leq C \oscH \,,
\end{equation}
for some $0<\alpha<1$ and where $C$ depends only on $n$, $\rho$ and $A_0$.
\end{corollary}

Hence, the lower bound on $\rho$ prevents any \emph{bubbling} phenomenon and corollary \ref{main2} quantifies the proximity of $S$ from a single bubble in a $C^1$ fashion. 

As far as we know, our results are the first quantitative studies for almost constant mean curvature hypersurfaces in the hyperbolic space. We mention that, in the Euclidean space, almost constant mean curvature hypersurfaces have been recently studied in \cite{CFSW,CFMN,CirMag,JEMS,KrMag,MagnPog}. In particular, theorem \ref{main} generalizes the results we obtained in \cite{JEMS} to the hyperbolic space. However, the generalization is not trivial. Indeed, even if a qualitative study of a problem via the method of moving planes in the hyperbolic space does not significantly differs from the Euclidean context, the quantitative study presents several technical differences which need to be tackled. 

Now we describe the proof of theorem \ref{main}. Here we work in the half-space model 
$$
\mathbb{H}^{n}=\{p=(p_1,\dots,p_n)\in\RR^n\,\,:\,\, p_n>0\}
$$ 
equipped with the usual metric
$$
g_p=\frac{1}{p_{n}^2}\,\sum_{k=1}^n\,dp_{k}\otimes dp_k\,.
$$

Our approach consists in a quantitative study of the method of the moving planes (for the analogue approach in the euclidean context see 
\cite{ABR,CFMN,CMS2,CMV,JEMS}).  Our first crucial result is to prove approximate symmetry in one direction. Indeed, we fix a direction $\omega$ and we perform the moving plane method along the direction $\omega$ until we get a \emph{critical} hyperplane $\pi_\omega$ (see subsection \ref{movingplannes} for a description of the method in the hyperbolic context). Possibly after applying an isometry we may assume $\pi_{\omega}$ to be the vertical hyperplane $\pi=\{p_1=0\}$. Hence $\pi$
intersects $S$ and the reflection of the right-hand cap of $S$ about $\pi$ is contained in $\Omega$ and is tangent to $S$. More precisely, let $S_+=S\cap \{p_1\geq 0\}$ and $S_-=S\cap \{p_1\leq 0\}$; then the reflection of $S_+$ about $\pi$ is contained in $\Omega$ and it is tangent to $S_-$ at a point $p_0$ (internally or at the boundary). If $A$ is a set, we denote by $A^{\pi}$ its reflection about $\pi$, and we will use the following notation:
$$
\mbox{$\hat \Sigma$ is the connected component of $S_-$ containing $p_0$}
$$
and 
$$
\mbox{$\Sigma$ is the connected component of $S_+^\pi$ containing $p_0$.}
$$
Furthermore, we denote by $N$ the inward normal vector field on $\Sigma$. The inward normal vector field on $\hat\Sigma$ is still denoted by $N$, since no confusion arises. We prove the following theorem on the approximate symmetry in one direction. 

\begin{theorem} \label{thm approx symmetry 1 direction}
There exists $\ep>0$ such that if
$$
{\rm osc}( H) \leq \ep,
$$
then for any $p \in\Sigma$ there exists $\hat p\in \hat\Sigma$ such that
\begin{equation*}
d(p,\hat p) +  |N_p-\tau_{\hat p}^p (N_{\hat p})|_p \leq C\, \oscH  . 
\end{equation*}
Here, the constants $\ep$ and $C$ depend only on $n$, $\rho$ and the area of $S$. In particular $\ep$ and $C$ do not depend on the direction $\omega$. 

Moreover, $\Omega$ is contained in a neighborhood of radius $C\oscH$ of $\Sigma \cup \Sigma^\pi$, i.e.
$$
d(p,\Sigma \cup \Sigma^\pi) \leq C\oscH \,,
$$
for every $p \in \Omega$.
\end{theorem}

In this last statement $\tau_p^q\colon \RR^n\to \RR^n$ denotes the parallel transport along the unique geodesic path in $\mathbb H^n$ connecting $p$ to $q$. We prove theorem \ref{thm approx symmetry 1 direction} by using quantitative tools for PDEs (like Harnack's inequality and quantitative versions of Carleson estimates and Hopf Lemma), as well as quantitative results for the parallel transport and graphs in the hyperbolic space.

In order to prove theorem \ref{main}, we first define an approximate center of symmetry $\mathcal{O}$ by applying the moving planes procedure in $n$ orthogonal directions. The argument here is not trivial, since $n$ ``orthogonal hyperplanes'' do not necessarily intersect, and theorem \ref{thm approx symmetry 1 direction} come into play.
Then, theorem \ref{thm approx symmetry 1 direction} is also used to prove that every critical hyperplane in the moving planes procedure is close to $\mathcal{O}$ and we finally prove estimates \eqref{stability radii} by exploiting theorem \ref{thm approx symmetry 1 direction} again.

\bigskip
\noindent\textit{Acknowledgments.}
The authors wish to thank Alessio Figalli, Louis Funar, Carlo Mantegazza, Barbara Nelli, Carlo Petronio, Stefano Pigola, Harold Rosenberg, Simon Salamon and Antonio J. Di Scala,  and  for their remarks and useful discussions. 
The first author has been supported by the \lq\lq Gruppo Nazionale per l'Analisi Matematica, la Probabilit\`a e le loro
Applicazioni\rq\rq (GNAMPA) of the Istituto Nazionale di Alta Matematica (INdAM) and the project FIR 2013 ``Geometrical and Qualitative aspects of PDE''.
The second author was supported by the project FIRB ``Geometria differenziale e teoria geometrica delle funzioni'' and by GNSAGA of INdAM.

\section{Preliminaries}\label{preliminar}
We recall some basic facts about the geometry of hypersurfaces in Riemannian manifolds. Let $(M,g)$ be an $n$-dimensional Riemannian manifold with Levi-Civita connection $\nabla$ and $i\colon S\hookrightarrow  M$ be an embedded orientable hypersurface of class $C^2$. Fix a unitary normal vector field $N$ on $S$. We recall that the {\em shape operator} of $S$ at a point $p\in S$ is defined  as
$$
W_p(v)=-\left(\nabla_{v}\tilde N_p\right)^{\perp}\in T_pS
$$
for  $v\in T_pS$, where $\tilde N$ is an arbitrary extension of $N$ in a neighborhood of $p$ and the upperscript \lq\lq $\perp$\rq\rq\, denotes the orthogonal projection onto $T_pS$.  $W_p$ is always symmetric with respect to $g$ and the {\em principal curvatures} $\{\kappa_1(p),\dots,\kappa_{n-1}(p)\}$ of $S$ at $p$ are by definition eigenvalues of $W_p$. We recall that the lowest and the maximal principal curvature at $p$ can be respectively obtained as the minimum and maximum of the map $\kappa_p\colon T_pS\backslash 0\to \RR$ defined as
$$
\kappa_p(v):=-\frac{1}{|v|^2}g_p(W_p(v),v)=-\frac{1}{|v|^2}g_p(\nabla_{v}\tilde N_p,v)\,.
$$
Alternatively, $\kappa_{p}(v)$ can be defined by fixing a smooth curve $\alpha\colon (-\epsilon,\epsilon) \to S$ satisfying
$$
\alpha(0)=p\,,\quad \dot \alpha(0)=v\,,
$$
since in terms of $\alpha$ we can write
$$
\kappa_p(v)=\frac{1}{|v|^2}g_p(N_p,D_t\dot\alpha(0))\,,
$$
where $D_t$ denotes the covariant derivative on $(M,g)$.  The {\em main curvature} of $S$ at $p$ is then defined as 
$$
H(p)=\frac{\kappa_1(p)+\dots +\kappa_{n-1}(p)}{n-1}\,. 
$$

From now on we focus on the hyperbolic space.  
Given a model of the hyperbolic space, we denote the hyperbolic metric by $g$, the hyperbolic distance by $d$, the hyperbolic norm at a point $p$ by $|\cdot |_p$, and the ball of center $p$ and radius $r$ by ${\sf B}_r(p)$. 
The Euclidean inner product in $\RR^n$ will be denoted by \lq\lq $\cdot$\rq\rq\, and the Euclidean norm by $|\cdot|$. The hyperbolic measure of a set $A$ will be denoted by $|A|_g$.

We mainly work in the half-space model $\mathbb H^n$. In this model hyperbolic balls and Euclidean balls coincide, but hyperbolic and  Euclidean centers and the hyperbolic and Euclidean radii differ. Namely, the Euclidean radius $r_E$ of ${\sf B}_r(p)$ is 
$$
r_E = p_n \sinh r \,,
$$
where $p=(p_1,\dots,p_n)$ are the coordinates of $p$ in $\mathbb{R}^n$. 

The Euclidean hyperplane $\{p_{n}=0\}\subset \RR^n$ will be denoted by $\pi_{\infty}$ and the origin of $\pi_{\infty}$ by $O$. Moreover, $\{e_1,\dots,e_n\}$ is the canonical basis of $\RR^n$. 

Given a point $p\in \H^{n}$, we denote by $\bar p$ its projection onto $\pi_{\infty}$ and by $B_{r}(x)$ the (Euclidean) ball of $\pi_{\infty}$ centered at $x\in \pi_{\infty}$ and having radius $r$. We omit to write the center of balls of $\pi_{\infty}$ when they are centered at the origin, i.e. $B_r(O)=B_r$.

\medskip 
Now we consider a closed $C^2$ hypersurface $S$ embedded in $\mathbb H^{n}$. 
Given a point $p$ in $S$ we denote by $T_pS$ its tangent space at $p$ and by $N_p$ the inward hyperbolic normal vector at $p$. Note that, accordingly to our notation,
$$
\nu_p:=\frac{1}{p_n}N_p
$$
is the Euclidean inward normal vector. 
We further denote by $d_S$ the distance on $S$ induced by the hyperbolic metric. Given a point $z_0 \in S$, we denote by $\mathcal B_r(z_0)$ the set of points on $S$ with intrinsic distance from $z_0$ less than $r$, i.e.
\begin{equation*}
\mathcal B_r(z_0)= \{z \in S:\ d_{S}(z,z_0) < r \}\,.
\end{equation*}

We are going to prove several quantitative estimates by locally writing the hypersurface $S$ as an Euclidean graph. Since this procedure is not invariant by isometries, we need to specify a ``preferred'' configuration in order to obtain uniform estimates. More precisely, such configuration is when $p=e_{n} \in S$ and $T_pS=\pi_{\infty}$; then, close to $p$, $S$ is locally the Euclidean graph  of a $C^2$-function $v\colon B_r\to \RR$ and we denote by $\mathcal{U}_r(p)$  the graph of $v$.  If $p$ in $S$ is an arbitrary point, then there exists an orientation preserving isometry $\varphi$ of $\mathbb H^n$ such that $\varphi(p)=e_n$  and  $T_{\varphi(p)}\varphi(S)=\pi_{\infty}$. 
Hence, around $\varphi(p)$, $\varphi(S)$ is the graph of a $C^2$-map $v\colon B_r\to \RR$ and we define $\mathcal U_r(p)$ as the preimage via $\varphi$ of the graph of $v$.  
The definition of $\mathcal U_r(p)$ is well-posed.

\begin{lemma}\label{Ur si well-posed}
The definition of $\mathcal U_r(p)$ does not depend on the choice of $\varphi$. 
\end{lemma}

\begin{proof}
Let $\mathcal U_r(p)$ be defined via an orientation-preserving isometry $\varphi\colon \mathbb H^n\to \mathbb H^n$ such that 
\begin{equation}\label{PHI}
\varphi(p)=e_n\,,\quad \varphi_{*|p}(T_pS)=\pi_{\infty}
\end{equation}
and let $\psi \colon \mathbb H^n\to \mathbb H^n$ be another  orientation-preserving isometry satisfying \eqref{PHI}.  Then $f=\psi \circ\varphi^{-1}$ is an  orientation-preserving isometry of $\mathbb H^n$ satisfying 
$$
f(e_n)=e_n\,,\quad f_{|*}(\pi_{\infty})=\pi_{\infty} 
$$
and so it is a rotation about the $e_n$-axis.  Therefore $\psi(\mathcal U_r(p))$ is the graph of a $C^2$-map defined on a ball in $\pi_{\infty}$ about the origin and the claim follows.  
\end{proof}

We denote by $H$ the hyperbolic mean curvature of $S$. $H$ is related to the Euclidean mean curvature $H_E$ by  
$$
H(p)=(\nu_{p}+p H_E(p)) \cdot e_n\,.
$$
For instance, if $S$ is the hyperbolic ball ${\sf B}_r(p)$ oriented by the inward normal, we have
$$
H\equiv \frac{1}{\tanh r} \,,\quad H_E(p)=\frac{1}{p_{n} \sinh r}\,.
$$
If $S$ is locally the graph of a smooth function $v\colon B_r\to \RR$, where $B_r$ is a ball about the origin in $\pi_\infty$,  and $p=(x,v(x)) \in S$, then $H$ at $p$ takes the following expression
\begin{equation}\label{Hu}
H(p)=\frac{v(x)}{n-1}\,{\rm div} \left(\frac{\nabla v(x)}{\sqrt{1+|\nabla v(x)|^2}}\right)+\frac{1}{\sqrt{1+|\nabla v(x)|^2}}\,. 
\end{equation}
In the last expression ${\rm div}$ and $\nabla$ are the Euclidean divergence and gradient in $\pi_{\infty}$, respectivily. Moreover,  we have 
$$
\nu_p=\frac{(-\nabla v(x),1)}{\sqrt{|\nabla v(x)|^2 + 1 }} \,.
$$

Since $S$ is compact and embedded, then  it is the boundary of a bounded domain  $\Omega$ in $\mathbb H^n$.
Given $p$ in $S$,  we say that {\em $S$ satisfies a touching ball condition of radius $\rho$ at $p$} if there exist two hyperbolic balls of radius $\rho$ tangent to $S$ at $p$, one contained in $\Omega$ and one contained in the complementary of $\Omega$. Since $S$ is compact then {\em $S$ satisfies a uniform touching ball condition of radius $\rho$} for some $\rho$, i.e. it satisfies a touching ball condition of radius $\rho$ at any point (see \cite{GT}).

\subsection{Alexandrov's theorem and the method of moving planes in the hyperbolic space}\label{movingplannes}
In this paper by  {\em hyperplane} in the hyperbolic space we mean a totally geodesic hypersurface.  In  the half-space model $\mathbb{H}^{n}$, hyperplanes are either Euclidean half-spheres centered at a point in $\pi_\infty$ or vertical planes orthogonal to  $\pi_{\infty}$, while in the ball model the hyperbolic hyperplanes are Euclidean spherical caps or planes orthogonal to the boundary of $\mathbb{B}^n$. Here we that recall the {\em ball model} consists of
$\mathbb{B}^n=\{p\in \RR^{n}\ |\ |p|=1\}$ equipped with the Riemannian metric
$$
g_p=\frac{4}{(1-|p|^2)^2}\sum_{k=1}^n dp_k\otimes dp_k\,.  
$$ 

If $\Omega$ is a bounded open set in the hyperbolic space, its {\em center of mass} is defined as the minimum point   
$\mathcal O$ of the map     
$$
P(p)=\frac{1}{2\,|\Omega|_g} \int_\Omega d(p,a)^2\,da \,.
$$
In view of \cite{Ka} $P$ is a convex function and the center of mass in unique. Furthermore the gradient of $P$ takes the following expression 
\begin{equation}\label{gradP}
\nabla P(p)=-\frac{1}{|\Omega|_g}\int_{\Omega} \exp^{-1}_p(a) \,da\,. 
\end{equation}
\begin{lemma} \label{lemma_dillovoi}
Let $\Omega$ be a bounded open set in the hyperbolic space.  Then every hyperplane of symmetry of $\Omega$ contains the center of mass $\mathcal O$ of $\Omega$. 
\end{lemma}
\begin{proof} Even if the result is well-known we give a proof for reader's convenience. 
We prove the statement in the ball model $\mathbb B^n$. Without loss of generality, we may assume that the center of mass $\mathcal O$ of $\Omega$ is the origin of $\mathbb B^n$. Assume by contradiction that there exists a hyperplane $\pi$ of symmetry for $\Omega$ not containing $\mathcal O$. Hence $\pi$ is a spherical cap which (up to a rotation) we may assume to be orthogonal to the line $(p_1,0,\dots,0)$ and lying in the half-space $p_1>0$.  Let $\pi_1=\{p_1=0\}$ be the vertical hyperplane orthogonal to  $e_1$. 
Since $\pi_1$ and $\pi$ are disjoint, they subdivide $\Omega$ in three subsets $\Omega_1$, $\Omega_2$, $\Omega_3$, with $|\Omega_2|_g>0$ (see figure \ref{figura2}). 
\begin{figure}[h]
$$
\begin{tikzpicture}
\shade (0,0) ellipse (60pt and 40pt);
\filldraw [black] (0,0) circle (1pt);
\node at (-2.3,2.5) {$\mathbb B^2$};
\node at (-1.9,1.2) {$\Omega$};
\node at (0.3,-0.3) {$\mathcal O$};
\node at (-0.8,0.5) {$\Omega_1$};
\node at (0.5,1) {$\Omega_2$};
\node at (1.6,0.5) {$\Omega_3$};
\node at (-0.4,2.5) {$\pi_1$};
\node at (1.8,2.5) {$\pi$};
\draw(0,0) ellipse (60pt and 40pt);
\draw(0,0) ellipse (80pt and 80pt);
\draw (0,2.8) -- (0,-2.8);
\draw[dashed] (2.8,0) -- (-2.8,0);
\draw [domain=216.1:144] plot ({4.8+4*cos(\x)}, {4*sin(\x)});
\end{tikzpicture}
$$
\caption{}
\label{figura2}
\end{figure}
Since $\Omega$ is symmetric about $\pi$, we have that $|\Omega_1|_g+|\Omega_2|_g=|\Omega_3|_g$. Moreover since
$$
\exp_{\mathcal O}^{-1}(p)=2(\tanh^{-1} |p|)\,\frac{p}{|p|}, \mbox{ for every } p\in \mathbb B^n \,,
$$
formula \eqref{gradP} implies 
$$
\int_{\Omega\cap \{p_1>0\}} (\tanh^{-1} |p|)\,\frac{p_1}{|p|} \,dp=-\int_{\Omega\cap \{p_1<0\}} (\tanh^{-1} |p|)\,\frac{p_1}{|p|} \,dp 
$$
so that  $|\Omega_1|_g=|\Omega_2|_g+|\Omega_3|_g$, which gives a contradiction.
\end{proof}

\begin{proposition} \label{prop_dai}
Let $S= \partial \Omega$ be a $C^2$-regular, connected, closed hypersurface embedded in the $n$-dimensional hyperbolic space, where $\Omega$ is a bounded domain. Assume that for every direction $\omega\in \mathbb R^n$ there exists a hyperplane of symmetry of $S$ orthogonal to $\omega$ at the center of mass $\mathcal O$ of $\Omega$. Then $S$ is a hyperbolic sphere about $\mathcal O$. 
\end{proposition}

\begin{proof}
We prove the statement in the ball model $\mathbb B^n$ assuming that  $\mathcal O$ is the origin  of $\mathbb B^n$. In this case the assumptions in the statement imply that $S$ is symmetric about every Euclidean hyperplane passing through the origin. So $S$ is an Euclidean ball about $\mathcal O$ (see e.g. \cite[Lemma 2.2, Chapter VII]{Ho}) and the claim follows.     
\end{proof}

Now we give a description of the method of the moving planes in $\mathbb H^n$ declaring some notation we will use here and in sections \ref{section_6} and \ref{section_proof_main2}.  The method consists in moving hyperbolic hyperplanes along a geodesic orthogonal to a fixed direction.    
Let $\omega$ be a fixed direction and let $\gamma_{\omega}\colon (-\infty,\infty)\to \mathbb H^{n}$ be the maximal geodesic satisfying $\gamma(0)=e_{n}$, $\dot\gamma(0)=\omega$. For any $s \in \mathbb{R}$ we denote by $\pi_{\omega,s}$ the totally geodesic hyperplane passing through $\gamma_\omega(s)$ and orthogonal to $\dot{\gamma}_\omega(s)$.

The description of the method can be simplified by assuming  $\omega=e_n$ (by using an isometry it is always possible to describe the method only for this direction). In this case the hyperplane $\pi_{e_n,s}$ consists of a half-sphere  $\pi_{e_n,s}=\{p \in \mathbb{H}^n:\, |p|={\rm e}^s\}$. For $s$ large enough, $S \subset \{|p| < {\rm e}^s\}$. We decrease the value of $s$ until $\pi_{e_n,s}$ is tangent to $S$. Then, we continue to decrease $s$ until the reflection $S_{e_n,s}^{\pi}$ of $S_{e_n,s}:=S \cap \{|p|\geq{\rm e}^{s}\}$ about $\pi_{e_n,s}$ is contained in $\Omega$, and we denote by $\pi_{e_n}$ the hyperplane obtained at the limit configuration. 

More precisely, for a general direction $\omega$ we define 
$$
m_{\omega}=\inf\{s \in \RR\,\,:\,\,S_{\omega,s}^\pi\subset \Omega\}
$$
and refer to $\pi_{\omega}:=\pi_{\omega,m_{\omega}}$ and $S_\omega:=S_{\omega,m_\omega}^{\pi}$ as to the {\em critical hyperplane} and \emph{maximal cap} of $S$ along the direction $\omega$. Analogously, $\Omega_\omega$ is addressed as the maximal cap of $\Omega$ in the direction $\omega$. Note that by construction the reflection $S_{\omega}^\pi$ of $S_\omega$ is tangent to $S$ at a point $p_0$ and there are two possible configurations given by $p_0 \not \in \pi_\omega$ and $p_0 \in \pi_\omega$.

\begin{proof}[Proof of Alexandrov's theorem]
The proof is obtained by using the method of the moving planes described above and showing that for every direction $\omega$ we have that $S$ is symmetric about $\pi_\omega$. Once a direction $\omega$ is fixed, we may assume by using a suitable isometry that $\pi_{\omega}$ is the vertical hyperplane $\pi_{\omega}=\{x_1=0\}$ and $\omega=e_1$. We parametrize $S$ and $S_\omega^\pi$ in a neighborhood of $p_0$ in $T_{p_0}S$ (which clearly coincides with $T_{p_0}S_\omega^\pi$) as graphs of two functions $v$ and $u$, respectively. If $p_0\notin \pi_{\omega}$ the functions $v$ and $u$ are defined on a ball $B_r$ (case (i)), otherwise they are defined in a half-ball $B_r \cap \{x_1 \leq 0\}$ and $v=u$ on $B_r \cap \{p_1 = 0\}$  (case (ii)). In both cases the two functions $v$ and $u$ satisfy \eqref{Hu} and the difference $w=u - v$ is nonnegative and satisfies an elliptic equation $Lw=0$,
with $w(0)=0$ in case (i) and $w=0$ on $B_r \cap \{p_1 = 0\}$ in case (ii). The strong maximum principle in case (i) and Hopf's lemma in case (ii) yield $w\equiv 0$. This implies that there exist two connected components of $S_-$ and $S_{\omega}^\pi$ such that the set of tangency points between them is both closed and open. Since $S$ is connected we also have that $S_{\omega}^\pi=S_-$, i.e. $S$ is symmetric about $\pi_\omega$. The conclusion follows from lemma \ref{lemma_dillovoi} and proposition \ref{prop_dai}.
\end{proof}

\begin{remark}{\em
We mention that Alexandrov's theorem still holds by assuming that $\Omega$ is connected, and the proof given above can be easily modified accordingly. 
}
\end{remark}

\begin{remark}{\em 
In the defintion of the method of the moving planes one can replace $e_n$ with an arbitrary point $p\in \mathbb H^n$ by replacing conditions $\gamma_\omega(0)=e_n$ and $\dot \gamma_{\omega}(0)=\omega$  with $\gamma_\omega(0)=p$ and $\dot \gamma_{\omega}(0)=\omega$, respectively.  }
\end{remark}

\begin{remark}{\em 
The method of the moving planes described in this section differs from the method of moving planes described in \cite{KP}, where the hyperplanes move along a horocycle instead of a geodesic. We remark that if one is interested in a qualitative result (such as the Alexandrov's theorem) then the two methods are equivalent; instead, the method we adopt here is more suitable for a quantitative analysis of the problem.}
\end{remark}

\section{Local quantitative estimates}

In this section we establish some local quantitative results that we need to prove theorem \ref{main}. We will need to switch Euclidean and hyperbolic distances and we need a preliminary lemma which quantifies their relation close to $e_n$.
We recall that the hyperbolic distance $d$ in the half-space model of $\mathbb H^n$ is given in terms of the  Euclidean distance by the following formula
\begin{equation} \label{dist_hyperbolic}
d(p,q) = \arccosh \left(1+\frac{|p-q|^2}{2p_{n} q_{n}} \right) \,.
\end{equation}
In particular 
$$
d(e_n,te_n)=|\log t|\,, \quad \mbox{ for any }t\in (0,\infty)\,. 
$$
%

\begin{lemma}\label{stima per d}
Let $R>0$ be fixed and let $q$ in ${\sf B}_R(e_n)$. Then there exist $c=c(R)>0$ and $C=C(R)>0$ such that
\begin{equation} \label{dist_equiv}
c |q-e_n| \leq d (q,e_n) \leq C |q-e_n| \,.
\end{equation}
\end{lemma}

\begin{proof}
Since $e^{-R} \leq q_n \leq e^R$, then
$$
1+\frac{e^{-R}}{2}|q-e_n|^2 \leq 1+ \frac{|q-e_n|^2}{2 q_{n}} \leq 1+\frac{e^{R}}{2}|q-e_n|^2 \,,
$$
and, since $|q-e_n| \leq e^R -1 $, then
$$
1+ \frac{|q-e_n|^2}{2q_{n}} \leq A \,,
$$
where $A=A(R)$. Let $\phi(t)=\arccosh(t)$, $t \in [1,+\infty)$. Since $1 \leq t \leq A$ then, keeping in mind that $
\phi'(t)=(t^2-1)^{-1/2}$, we have 
$$
\frac{1}{\sqrt{A+1}} \frac{1}{\sqrt{t-1}} \leq \phi'(t) \leq  \frac{1}{\sqrt{t-1}} \,,
$$
and hence
$$
\frac{1}{2\sqrt{A+1}} \sqrt{t-1} \leq \phi(t) \leq \frac{1}{2} \sqrt{t-1} \, \quad t \in [1,A] \,.
$$
By letting
$$
t=1 + \frac{|q-e_n|^2}{2 q_{n}} \,,
$$
and from
$$
\frac{e^{-R/2}}{\sqrt{2}} |q-e_n| \leq \sqrt{t-1} \leq \frac{e^{-R/2}}{\sqrt{2}} |q-e_n|
$$
we conclude.
\end{proof}

\subsection{Quantitative estimates for parallel transport}
In this subsection we prove quantitative estimates involving the parallel transport which will be useful in the proof of theorem \ref{thm approx symmetry 1 direction}.  

We recall that {\em the parallel transport} along a smooth curve  $\alpha\colon [t_0,t_1]\to \mathbb H^n$  is the linear map $\tau \colon \RR^n\to \RR^n$ given by 
$$
\tau(v)=X(t_1)
$$
where $X\colon [t_0,t_1]\to \mathbb H^n $ is the solution to the linear ODE 
$$
\begin{cases}
\dot X_k+\sum_{i,j=1}^nX_j\dot \alpha_i\Gamma_{ij}^k(\alpha)=0\,,\quad k=1,\dots, n,\\
X_k(t_0)=v_k,\quad k=1,\dots, n\,,
\end{cases}
$$
and $\Gamma_{ij}^k$ are the Christoffel symbols in $\mathbb{H}^n$. Here we recall that the $\Gamma_{ij}^k$'s are all vanishing if either the three indexes $i,j,k$ are distinct or one of them is different from $n$, while in the remaining cases they are given by
$$
\Gamma_{in}^i=-\frac{1}{x_n}\,\quad \Gamma_{ii}^n=\frac{1}{x_n}\,,\quad \Gamma_{ni}^i=-\frac{1}{x_n}\,,\quad \Gamma_{nn}^n=-\frac{1}{x_n}\,.
$$
We adopt the following notation: given $q$ and $p$ in $\mathbb H^n$, we denote by 
$$
\tau_{q}^p\colon \mathbb R^n\to \mathbb R^n
$$ 
 the parallel transport along the unique geodesic path connecting $q$ to $p$. Note that if $q$ and $p$ belong to the same vertical line (i.e. if $\bar q=\bar p$ in our notation), then 
$$
\tau_q^p(v)=\frac{p_n}{q_n}\,v\,. 
$$
About the case, $\bar q\neq \bar p$, we consider the following lemma where for simplicity we assume $p=e_n$. 

\begin{lemma}\label{computation}
Let $q\in\mathbb H^n$ be such that $q\in \langle e_{n-1},e_n\rangle$ and let  $v\in \RR^n$. Assume $q_{n-1}\neq 0$, then 

$$
\tau_q^{e_n}(v)=\frac{1}{q_n}(v_1,\dots, v_{n-2},\tilde v_{n-1},\tilde v_n)\,,
$$ 
where 
$$
\left(\begin{array}{c}\tilde v_{n-1}\\
\tilde v_{n}
\end{array}
\right)=\frac{1}{1+a^2}
\left(\begin{array}{cc} a(a-q_{n-1})+q_n& a-q_{n-1}-aq_n\\
cq_n-a+q_{n-1} &  a(a-q_{n-1})+q_n
\end{array}
\right)
\left(\begin{array}{c} v_{n-1}\\
v_{n}
\end{array}
\right)
$$
and 
$$
a=\frac{|q|^2-1}{2q_{n-1}}\,. 
$$

\end{lemma}

\begin{proof}
Let $\alpha\colon [t_0,t_1]\to \mathbb H^n$ be defined as  
$$
\alpha(t)=(R\cos(t)+a)e_{n-1}+R\sin(t)e_n\,,
$$
where 
$$
a=\frac{|q|^2-1}{2q_{n-1}}\,,\quad R=\sqrt{1+a^2}
$$
and 
$$
\alpha(t_0)=q\,,\quad \alpha(t_1)=e_n\,. 
$$
Then $\alpha$,  up to be parametrized, is a geodesic path connecting $q$ to $e_n$. The parallel transport equation along $\alpha$ yields 
$$
(\tau_q^{e_n}(v))_k=v_k\,,\quad k=1,\dots,n-2\,, 
$$
while 
$$
(\tau_q^{e_n}(v))_{n-1}=X_{n-1}(t_1)\,,\quad (\tau_q^{e_n}(v))_{n}=X_{n}(t_1)\,, 
$$
where the pair $(X_{n-1},X_n)$ solves 
$$
\left(\begin{array}{l}\dot X_{n-1}\\
\dot X_n
\end{array}
\right)=\left(\begin{array}{cc} {\rm cotan}\,t & -1\\
1 & {\rm cotan} \,t 
\end{array}
\right)\left(\begin{array}{l} X_{n-1}\\
X_n
\end{array}
\right)\,,\quad 
\left(\begin{array}{l} X_{n-1}(t_0)\\
X_n(t_0)
\end{array}
\right)=
\left(\begin{array}{l}v_{n-1}\\
v_n
\end{array}
\right)\,.$$
Therefore 
$$
\left(\begin{array}{l} X_{n-1}(t)\\
X_n(t)
\end{array}
\right)=A(t)A(t_0)^{-1}\left(\begin{array}{l} v_{n-1}\\
v_n
\end{array}
\right)\,,\quad A(t):=\left(\begin{array}{cc} \cos t\,\sin t & -\sin^2 t\\
\sin^2 t &  \cos t\,\sin t 
\end{array}
\right)
$$
and the claim follows. 
\end{proof}
The following two propositions give some quantitive estimates involving the map $\tau_{q}^p$.   
\begin{proposition} \label{STP1}
Let $p$ and $q$ in $\mathbb H^n$ and let $\omega$ be the global vector field  $\omega_z=z_ne_1$.  Then 
$$
|\omega_p-\tau_q^p(\omega_q)|_p\leq C\, d(p,q)\,,
$$
where $C$ depends on an upper bound on the distance between $p$ and $q$. 
\end{proposition}

\begin{proof} 
Note that in the simple case where $p$ and $q$ belong to the same vertical line, then  the claim is trivial since $|\omega_p-\tau_q^p(\omega_q)|_p=0$. We focus on the other case. 
Let $f\colon \mathbb H^n\to \mathbb H^n$ be 
$$
f(z)=\frac{1}{p_n}\mathcal {R}\left(z-\bar p\right)
$$
where $\mathcal R$ is a rotation around the $e_n$-axis such that 
$$
\mathcal {R}\left(q-\bar p\right)\in \langle e_{n-1},e_{n}\rangle\,.
$$  
In this way we have 
$$
f(p)=e_{n}\,,\quad f(q)\in \langle e_{n-1},e_n\rangle\,,\quad f_{|_*z}(\omega_z)=f(z)_n\,v \,\,\,\mbox{ for all }z\in \mathbb H^n\,, 
$$
where $v=\mathcal R(e_1)$. We set $f(q)=\hat q$  and we write $\hat q=\hat{q}_{n-1}e_{n-1}+\hat{q}_{n}e_n$. Now $\hat q_{n-1}\neq 0$ and  we can apply lemma \ref{computation} obtaining 
$$
\tau_{\hat q}^{e_n}(\hat q_{n} v)=\left(v_1,\dots, v_{n-2},\frac{1}{1+a^2}(a(a-\hat q_{n-1})+\hat q_n)v_{n-1},\frac{1}{1+a^2}(a\hat q_n-a+\hat q_{n-1})v_{n-1}\right)\,,
$$
where 
$$
a=\frac{|\hat q|^2-1}{2\hat q_{n-1}}\,. 
$$
Furthermore a direct computation gives 
$$
|v-\tau_{\hat q}^{e_n}(\hat q_{n} v)|=\frac{|v_{n-1}|}{\sqrt{1+a^2}}\,|\hat q-e_n|\,.
$$
Since $|v|=1$, keeping in mind lemma \ref{stima per d}, we have 
$$
|\omega_p-\tau_q^p(\omega_q)|_p=|v-\tau_{\hat q}^{e_n}(\hat q_{n} v)|=\frac{|v_{n-1}|}{\sqrt{1+a^2}}\,|\hat q-e_n|\leq \frac{1}{c}\, d(e_n,\hat q)=\frac{1}{c} d(p,q)\,,
$$
where $c$ is a small constant depending on $d(e_n,\hat q)=d(p,q)$. Hence the claim follows. 
%
%
\end{proof}

\begin{proposition} \label{STP2}
Let $q$, $\hat q$ and $z$ in $\mathbb H^n$ and $R>0$ be such that 
$$
q,\hat q\in {\sf B}_{R}(z)\,. 
$$
Let $v,w\in \RR^n$ be such that 
$$
|v|_q=|w|_{\hat q}=1\,.
$$
Then 
$$
|\tau_{q}^z(v)-\tau_{\hat q}^z(w)|_z\leq C\left(d(z,q)+d(z,\hat q)+d(q,\hat q)+|v-\tau_{\hat q}^{q}(w)|_{ q} \right)
$$
where $C$ is a constant depending only on $R$.  
\end{proposition}
\begin{proof}
We first consider the case where the three points $q,\hat q, z$ belong to the same geodesic path. In this case we may assume that $z=e_n$ and that $q$ and $\hat q$ belong to the $e_n$ axis, i.e. 
$$
q=q_n\,e_n\quad \mbox{ and } \quad \hat q=\hat q_n\,e_n\,. 
$$
Under these assumptions  we have
$$
|\tau_{q}^z(v)-\tau_{\hat q}^z(w)|_z=\left|\frac{1}{q_n}v-\frac{1}{\hat q_n} w\right|=|v-\tau_{\hat q}^{q}(w)|_{ q}
$$
and the claim is trivial. 
Next we focus on the case where the three points do not  belong to the same geodesic path. Up to apply an isometry, we may assume:   
$z=e_n$, $q$ and $\hat q$ belonging to the same vertical line and $z,q,\hat q$ belonging to the plane $\langle e_{n-1},e_n\rangle$. Note that $q_{n-1}=\hat q_{n-1}\neq 0$. 
In the next computation we denote by $\| \cdot\|$ the norm of linear operators $\RR^n\to \RR^n$ with respect to the  Euclidean norm. Note that
$$
\|\tau_{ q}^z\|=\frac{1}{q_n}\,,\quad \|\tau_{ \hat q}^z\|=\frac{1}{\hat q_n}\,,\quad |v-\tau_{\hat q}^{q}(w)|_{ q}=\left|\frac{1}{q_n}v-\frac{1}{\hat{q}_n}w\right|\,.
$$
Taking into account that $|v|=q_n$ and $|w|=\hat q_n$, we have 
$$
\begin{aligned}
|\tau_{q}^z(v)-\tau_{\hat q}^z(w)|_z \leq & \left|1-\frac{1}{q_n}\right|\, |\tau_{q}^z(v)|+\left|\tau_{q}^z\left(\frac{1}{q_n}v-\frac{1}{\hat{q}_n}w\right)\right|+\frac{1}{\hat{q}_n}|\tau_{q}^z(w)-\tau_{\hat q}^z(w)|+
\left|\frac{1}{\hat{q}_n}-1\right|\,|\tau_{\hat q}^z(w)|\\
\leq &|q_n-1|\|\tau_{ q}^z\|+\|\tau_{ q}^z\|\,\left|\frac{1}{q_n}v-\frac{1}{\hat{q}_n}w\right|+\|\tau_{q}^z-\tau_{\hat q}^z\|+|\hat q_n-1|\|\tau_{ \hat q}^z\|\\
=& \frac{1}{q_n}\left( |q_n-1|  +  |v-\tau_{\hat q}^{q}(w)|_{ q} \right) + \frac{|\hat q_n-1|}{\hat q_n}+ \|\tau_{q}^z-\tau_{\hat q}^z\| \,.
\end{aligned}
$$
From lemma \ref{computation}, we have that $\|\tau_{q}^z-\tau_{\hat q}^z\| \leq C d(q,\hat q)$, where $C$ is a constant depending only on $R$, and from lemma \ref{stima per d} we conclude.
\end{proof}

\subsection{Local quantitative estimates for hypersurfaces}
In this subsection we prove some quantitative estimates for hypersurfaces in the hyperbolic space.

Throughout this subsection, $S$ denotes a $C^2$-regular closed hypersurface embedded in $\mathbb H^n$ satisfying a uniform touching ball condition of radius $\rho$. We notice that the hyperbolic ball of radius $\rho$ centred at $q=(\bar q,q_n)$ of radius $\rho$ is the Euclidean ball of radius $q_n\sinh(\rho)$ centred at $(\bar q, q_n	\cosh \rho)$.   

Furthermore we set 
\begin{eqnarray}
&& \rho_0= e^{-\rho} \sinh \rho\,,\label{rho0}\\
&& \rho_1= (1-\rho_0) \rho_0\,. \label{rho1}
\end{eqnarray}
Notice that $\rho_0$ is the Euclidean radius of a hyperbolic ball of radius $\rho$ with center at $(0,\ldots,0,e^{-\rho})$. Therefore if $e_n$ belongs to $S$, then $S$ satisfies an Euclidean touching ball condition of radius $\rho_0$ at $e_n$.

Note that, since $S$ satisfies a uniform touching ball condition of radius $\rho$, every geodesic ball $\mathcal B_{r}(p)$ of radius $r\leq \rho_0$ in $S$ is such that 
\begin{equation}\label{geodesic_disk}
|\mathcal B_{r}(p)|\geq cr^{n-1}\,,
\end{equation}
where $c$ depends only on $n$.  The inequality can be easily proved assuming $p=e_{n}$ and $T_pS=\pi_{\infty}$ and then applying lemma \ref{stima per d}. 

\begin{lemma}\label{fejahyp}
Assume $e_n\in S$ and 
$T_{e_n}S=\pi_{\infty}$. Then $S$ can be locally written around $e_n$ as the graph a $C^2$-function  $v\colon B_{\rho_1}\subset \pi_{\infty}\to \RR$,
satisfying
\begin{equation} \label{bounds on u}
v(O)=1\,,\quad |v(x)-1| \leq \rho_1 - \sqrt{\rho_1^2-|x|^2}\,,\quad |\nabla v(x)| \leq \frac{|x|}{\sqrt{\rho_1^2-|x|^2}}\,
\end{equation}
for every $x \in B_{\rho_1}$. 
\end{lemma}

\begin{proof}
Since $S$ satisfies a touching ball condition of radius $\rho$, then any point $q\in S\cap (B_{\rho_0}\times (1-\rho_0,1+\rho_0))$ satisfies an Euclidean touching ball condition of radius $\rho_1$. The claim then follows from \cite[lemma 2.1]{JEMS}.
\end{proof}

Note that accordingly to the terminology introduced in the first part of section \ref{preliminar}, the graph of the map $v$ in the statement above is denoted by $\mathcal U_{\rho_1}(e_n)$.   

\begin{proposition}\label{fejahyp2}
There exists $\delta_0=\delta_0(\rho)$ such that if $p,q\in S$ with $d_S(p,q)\leq \delta_0$ then
\begin{equation} \label{bound on nu N+1}
g_p(N_p, \tau_q^p(N_q))\geq  \sqrt{1-C^2d_S(p,q)^2}\ \ \quad  \textmd{ and } \quad \ \ |N_p - \tau_q^p(N_q)|_p \leq C d_S(p,q)\,,
\end{equation}
where $C$ is a constant depending only on $\rho$.
\end{proposition}

\begin{proof}
%
%
%
%
We will choose $\delta_0=\min(r_2,1/C)$, see below for the definition of $r_2$ and $C$. 

Possibly after applying an isometry, we can assume that $p=e_{n}$ and $q=t e_{n}$. 
We notice that any point in $S$ which is far from $e_n$ less than $\rho$ satisfies an Euclidean touching ball condition of radius $r_1$, where $r_1$ depends only on  $\rho$. Moreover from lemma \ref{stima per d}, there exists $0<r_2=r_2(\rho)$ such that if $d(e_{n},q)\leq r_2$ then $|e_{n}-q| \leq r_1/2$; this implies that, being 
$$
d(p,q) \leq d_S(p,q)\leq r_2\,,
$$ 
we have 
$$
|1-t|=|p-q| \leq r_1/2 \,.
$$
Now we can apply  the Euclidean estimates in \cite[lemma 2.1]{JEMS} to $p$ and $q$ (with $r_1$ in place of $\rho$) and we obtain 
$$
\nu_p \cdot \nu_q \geq \sqrt{1-\frac{|p-q|^2}{r_1^2}} \,.
$$
Since $d(p,q)\leq \rho$, from \eqref{dist_equiv} we have that $|p-q|\leq C_1 d(p,q) \leq C_1 d_S(p,q)$ for some constant $C_1=C_1(\rho)$, and hence
\begin{equation}\label{intermediate}
\nu_p \cdot \nu_q \geq \sqrt{1-C^2d_S(p,q)^2} \,,
\end{equation}
where $C=C_1/r_1$ and provided that $d_S(p,q) < 1/C$. Since
$$
N_p=\nu_p\,,\quad \nu_q=\frac{1}{t}N_q=\tau_q^p(N_q) \,,
$$
inequality \eqref{intermediate} can be written as
$$
g_p(N_p, \tau_q^p(N_q))\geq  \sqrt{1-C^2d_S(p,q)^2} \,,
$$
which is the first inequality in \eqref{bound on nu N+1}. The second inequality in \eqref{bound on nu N+1} follows by a direct computation.
\end{proof}

\begin{lemma} \label{lemma bound on d Ga}
For any $0<\alpha<\frac{1}{2}\min(1,\rho_1^{-1})$ there exists a universal constant $C$ such that if $q \in \mathcal{U}_{\alpha\rho_1}(p)$, then
\begin{equation}
 \label{disk}  d_S(p,q) \leq \alpha C\rho_1
\end{equation}
and
\begin{equation}
 \label{bound on d Ga above}  d(p,q) \leq d_S(p,q) \leq C \cosh(\rho_1) d(p,q) \,.
\end{equation}

\end{lemma}

\begin{proof}
Possibly after applying an isometry, we can assume that $p=e_n$ and $\nu_p=e_{n}$.
Lemma \ref{fejahyp} implies that $S$ is the graph of a $C^2$ function $v\colon B_{\rho_1} \to \RR$.  Let $q=(x,v(x))$ with $|x| < \rho_1$ (so that  $q \in \mathcal{U}_{\rho_1}(p)$) and
consider the curve  $\gamma\colon [0,1]\to \mathcal{U}_{\rho_1}(p)$ joining $p$ with $q$ defined by   $\ga(t)=(tx,v(tx))$. Then
$$
\dot\gamma(t)=(x,\nabla v(tx)\cdot x)\,.
$$
The Cauchy-Schwartz inequality implies
$$
|\dot\gamma(t)|\leq |x| \sqrt{1+|\nabla v(tx)|^2}\,.
$$
Therefore inequality  \eqref{bounds on u}  in lemma \ref{fejahyp} implies
$$
 |\dot\gamma(t)|\leq \,\frac{\rho_1 |x|}{\sqrt{\rho_1^2-t^2|x|^2}} \leq \frac{|x|}{\sqrt{1-\alpha^2}} \leq\frac{2}{\sqrt{3}}|x|,
$$
for $0 \leq |x| \leq \alpha \rho_1$.
Since
\begin{equation*}
d_S(p,q) \leq \int_{0}^1 \frac{|\dot\gamma(t)|}{v(tx)} dt \, ,
\end{equation*}
and from \eqref{bounds on u} we obtain that
$$
d_S(p,q) \leq C |x|
$$
for some universal constant $C$, which implies \eqref{disk}. Being
$$
|x|  \leq |p-q|, 
$$
a careful analysis of the constant appearing in \eqref{dist_equiv} gives \eqref{bound on d Ga above}. 
\end{proof}


\begin{lemma} \label{lemma change normal I}
Assume $p=te_n \in S$, for some $t\in [1,\infty)$ and $\nu_p$ be such that
$$
\nu_p\cdot e_{n}>0\,,\quad |\nu_p-e_{n}|\leq \ep \,,
$$
for some $0\leq \ep < 1$. Then, in a neighborhood of $p$, there exists a $C^{2}$-function $v: B_{r}\to \RR$, with $r=\rho_1 \sqrt{1-\ep^2}$, such that $p=(0,v(0))$ and $S$ is locally the graph of $v$.
\end{lemma}
 
\begin{proof}
Notice that if $d_S(p,q)\leq \log(1-\rho_0)$, then $q_n\geq 1-\rho_0$ and $q$ satisfies an Euclidean touching ball condition of radius $\rho_1$. The claim then follows from the Euclidean case, see  \cite[lemma 3.4]{JEMS}. 
\end{proof}

\section{Curvatures of projected surfaces} \label{subsect Luigi}

In order to perform a quantitive study of the method of the moving planes, we need to handle the following situation: given a hypersurface $U$  of class $C^2$ in $\mathbb H^{n}$, we consider its intersection $U'$ with a hyperbolic hyperplane $\pi$. If $\pi$ intersects $U$ transversally, $U'=U\cap \pi$ is a hypersurface of class $C^2$ of $\pi$ and we consider its Euclidean orthogonal projection $U''$ onto $\pi_\infty$
(see figure \ref{figura1} for an example in $\mathbb H^3$).
\begin{figure}[h] 
\includegraphics[width=0.6\textwidth]{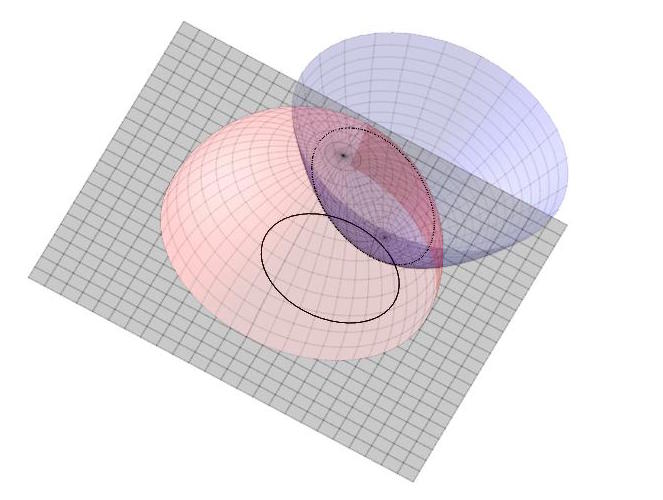}  
\caption{In the figure $U$ is the parabololid in $\mathbb H^3$ parametrized by $\chi(u,v)=(v\cos(u),1/2-v\sin(u),v^2+1/2)$
and  $\pi$ is the half-sphere about the origin of radius one.} \label{figura1}
\end{figure}
The next  propositions allow us to control the Euclidean principal curvature of $U''$ in terms of the hyperbolic principal curvature of $U$.
\begin{proposition}\label{prop Luigi I}
 Let $U$ be a $C^2$-regular embedded  hypersurface in $\mathbb H^n$ oriented by a unitary normal vector field $N$.  Let $\kappa_j$, $j=1,\ldots,n-1$, be the principal curvatures of $U$ ordered increasingly, $\pi$ be a hyperplane in $\mathbb H^n$ intersecting $U$ transversally and  $U'=U\cap \pi$. Then $U'$ is  an orientable hypersurface of class $C^2$ embedded in $\pi$ and, once a unitary normal vector filed $N'$ on $U'$ in $\pi$ is fixed, its principal curvatures $\kappa_i'$ satisfy
\begin{equation} \label{bound curv I}
\frac{1}{g_q( N_q,N_q')}\kappa_1(q)\leq \kappa'_i(q)\leq \frac{1}{ g_q( N_q,N_q') }\kappa_{n-1}(q) 
\end{equation}
for every $q\in U'$ and $i=1,\dots,n-2$. Furthermore, once a unitary normal  vector field $\omega$ on $\pi$ is fixed, we have
\begin{equation} \label{bound curv I'}
\frac{1}{\sqrt{1-g_q(\omega_q,N_q)^2}}\kappa_1(q)\leq \kappa'_i(q)\leq \frac{1}{ \sqrt{1-g_q(\omega_q,N_q)^2}}\kappa_{n-1}(q) \,,
\end{equation}
for every $q\in U'$ and a suitable choice of $N'$.
\end{proposition}

\begin{proof}
Up to apply an isometry, we may assume that $\pi$ is the vertical hyperplane $\{p_1=0\}$.

First observe that  $U'$ is of class $C^2$ by the implicit function theorem and it is orientable since 
\begin{equation} \label{nu primo}
N'_q=(-1)^{n}\frac{*(*(\nu_q\wedge \partial_{x_1})\wedge \partial_{x_1})}{|*(*(\nu_q\wedge \partial_{x_1})\wedge \partial_{x_1})|_q}
\end{equation}
defines a unitary normal vector field on $U'$, where $\nu_q=\frac{1}{q_n}N_q$ is the Euclidean normal vector filed on $U$ and  $*$ is the Euclidean Hodge star operator in $\RR^{n}$.

In order to prove \eqref{bound curv I}: fix $q\in U'$ and consider a vector  $v\in T_qU'$  satisfying $|v|_q=1$. Set 
$$
\kappa_q(v)=g_q(\nabla_{v}\tilde N,v)\,,
$$
where $\tilde N$ is an arbitrary extension of $N$ in a neighborhood of $q$ and $\nabla$ is the Levi-Civita connection of $g$.
Since $N_q$ is orthogonal to $T_qU'$,  it belongs to the plane generated by $\partial_{x_1}$ and $N'_q$ and we can write
$$
N=a\, \partial_{x_1}+ b N'\,,
$$
where
$$
b=g(N,N')\,.
$$
Let $\tilde a$, $\tilde b$ and $\tilde N'$ be arbitrary extensions of $a$, $b$ and $N'$ in the whole $\mathbb H^{n}$. Therefore
$$
\tilde N=\tilde a\,  \partial_{x_1}+ \tilde b\,\tilde N'
$$
is an extension of $N$.  We have

$$
\begin{aligned}
\kappa_q(v)=\,&g_q(\nabla_{v}\tilde N,v)=g_q(\nabla_{v}(\tilde a\,  \partial_{x_1}+ \tilde b\,\tilde N'),v)\\
=&v(\tilde a)\,g_q(\partial_{x_1},v)+v(\tilde b)\,g_q(N'_q,v)+ a(q)\,g_q(\nabla_{v} \partial_{x_1},v)
+ b(q)\,g_q(\nabla_{v} \tilde N',v)\\
=&a(q)\,g_q(\nabla_{v} \partial_{x_1},v)+b(q)\,g_q(\nabla_{v} \tilde N',v)\,. 
\end{aligned}
$$
Since $\pi$ is a totally geodesic submanifold $g_q(\nabla_{v} \partial_{x_1},v)=0$, and  therefore
$$
\kappa_q(v)=g_q(N_q,N'_q)\,g_q(\nabla_{v} \tilde N',v)
$$
which implies \eqref{bound curv I}.

Now we prove \eqref{bound curv I'}. Let $\nu'_{q}=\frac{1}{q_n} N'_q$. Then $\nu'$ is an Euclidean unitary normal vector field on $U'$ and a standard computation yields 
$$
\nu_q\cdot \nu'_q=1-(\nu_q\cdot e_1)^2
$$
(see e.g. \cite[section 2.3]{JEMS}). Therefore, if $\omega_q=q_ne_1$, then 
$$
g_q( N_q,N_q')=\nu_q\cdot \nu'_q=1-(\nu_q\cdot e_1)^2=1-g_q(N_q,\omega_q)^2
$$
and \eqref{bound curv I'} follows. 
\end{proof}

Note that in the statement of proposition \ref{prop Luigi I}, the $\kappa_i'$ are the curvatures of $U'$ once it is considered a hypersurface of $\pi$ and not when it is seen as hypersurface of $U$. A bound on the principal curvatures of $U'$ as hypersurface in $U$ is given by the following proposition.   

\begin{proposition}\label{prop Luigi III}
Under the same assumptions of proposition \ref{prop Luigi I}, the principal curvatures ${\check \kappa}_i'$ of $U'$ seen  as a hypersurface of $U$ satisfy
$$
|\check \kappa'_i(q)|\leq \frac{|g_q(\omega_q,N_q)|}{\sqrt{1-g_q(\omega_q,N_q)^2}}\,\max\{|\kappa_1(q)|,|\kappa_{n-1}(q)|\}\,,
$$
where $\omega$ is a normal unitary vector field to  $\pi$. 
\end{proposition}
\begin{proof}
We prove the statement assuming $\pi$ to be the vertical hyperplane $\{p_1=0\}$ and $\omega_p=p_n e_1$, for $p\in \pi$. 
Let $q\in U'$, $v\in T_qU'$ such that $|v|_q=1$ and let $\alpha\colon (-\delta,\delta)\to S$ be a unitary speed curve satisfying $\alpha(0)=q$, $\dot{\alpha}(0)=v$. Fix a unitary normal vector field $\tilde N'$ of $U'$ in $U$ near $q$.   We may complete $v$ with an orthonormal basis $\{v,v_2,\dots v_{n-2}\}$ of $T_qU'$ such that 
$$
\check N'_q=*_q(N_q\wedge v\wedge v_2\wedge\dots \wedge v_{n-2})\,,
$$ 
where $*_q$ is the Hodge star operator at $q$ in $\mathbb H^n$ with respect to $g$ and the standard orientation. Set 
$$
\check \kappa'_q(v)=g_{q}(*_q(\check N_q\wedge v\wedge v_2\wedge\dots \wedge v_{n-2}),D_t\dot \alpha_{|t=0})\,,
$$
where $D_t$ is the covariant derivative in $\mathbb H^n$.
Since $D_{t}\dot\alpha_{|t=0}\in \pi$, we have 
$$
\check \kappa'_q(v)=g_q(N_q,\omega_q)g_{q}(*_q(\omega_q\wedge v\wedge v_2\wedge\dots \wedge v_{n-2}),D_t\dot \alpha_{|t=0})\,. 
$$
Now, $*_q(\omega_q\wedge v\wedge v_2\wedge\dots \wedge v_{n-2})$ is a normal vector to $T_qU'$ in $\pi$ and so  
$$
\check \kappa'_q(v)=g_q(N_q,\omega_q)g_q(\nabla_{v}\tilde N,v)\,,
$$
where $\tilde N$ is an arbitrary extension of $N$ in a neighborhood of $q$. Proposition \ref{prop Luigi I} then implies 
$$
|\check \kappa'_q(v)|\leq \frac{|g_q(N_q,\omega_q)|}{\sqrt{1-g_q(\omega_q,N_q)^2}}\,\max\{|\kappa_1(q)|,|\kappa_{n-1}(q)|\}\,,
$$
as required. 
\end{proof}
%
Before giving the last result of this section, we recall the following notation introduced in the first part of the paper: given a point $q\in \mathbb{H}^{n}$, we denote by $\bar q $ its orthogonal projection onto $\pi_{\infty}$, i.e.
$$
q=(\bar q,q_{n})\,.
$$

\begin{proposition} \label{prop Luigi II}
Let $\pi$ be a non-vertical hyperplane in $\mathbb H^{n}$ and $U'$ be a $C^2$ regular hypersurface of $\pi$ oriented by a unitary normal vector field $N'$ in $\pi$.
Denote by $\kappa'_i$, for $i=1,\dots, n-2$,  the principal curvatures of $U'$.
Then the Euclidean orthogonal projection $U''$ of  $U'$ onto $\pi_{\infty}$ is a $C^2$-regular hypersurface of $\pi_{\infty}$ with a canonical orientation.
Moreover, for any $q\in U'$ we have
\begin{equation}\label{bound curvatures II}
|\kappa_i''(\bar q)| \leq \frac{1}{R}
 \left((\nu_q'\cdot e_{n})^2+\dfrac{q_{n}^2}{R^2}\right)^{-3/2}
\left(\max\{|\kappa_1'(q)|,|\kappa_{n-2}'(q)|\}+3\right),
\end{equation}
for every $i=1,\dots,n-2$,
where $\{\kappa''_i\}$ are the principal curvature of $U''$ with respect to the Euclidean metric and $R$ is the Euclidean Radius of $\pi$ and $\nu'_q=\frac{1}{q_n}N'_q$. 
\end{proposition}
\begin{proof}
By our assumptions, $\pi$ is a half-sphere of radius $R$ with center in $\pi_\infty$. By considering a suitable isometry, we may assume that $\pi$ has center at the origin of $\pi_\infty$.
If $X$ is a local positive oriented parametrisation of $U'$, then $\bar X=X-(X\cdot e_{n})e_{n}$ is a local parametrisation of $U''$, and we can orient $U''$ with
\begin{equation}\label{defnu2}
\nu''\circ \bar X:={\rm vers}(*(\bar X_1\wedge \bar X_2\wedge \dots \wedge \bar X_{n-2}\wedge e_{n}))\,,
\end{equation}
where $\bar X_k$ is the $k^{th}$ derivative of $\bar X$ with respect to the coordinates of its domain and $*$ is the Hodge \lq\lq star\rq\rq operator in $\RR^{n}$ with respect to the the Euclidean metric and the standard orientation. Therefore $U''$ is a $C^2$-regular hypersurface of $\pi_\infty$ oriented by the map $\nu''$.

Now we prove inequalities \eqref{bound curvatures II}. Fix a point $q=(\bar q,q_{n})\in U'$ and $\bar v\in T_{\bar q}U'$ be nonzero.
Let $\beta \colon (-\delta,\delta)\to U''$ be an arbitrary regular curve contained in $U''$ such that
$$
\beta(0)=\bar q\,,\quad \dot\beta(0)=\bar v\,.
$$
Then
$$
\kappa''_{\bar q}(\bar v)=\frac{1}{|\bar v |^2}\nu''_{\bar q}\cdot \ddot \beta(0)
$$
is the normal curvature of $U''$ at $(\bar q,\bar v)$, viewed as hypersurface of $\pi_{\infty}$ with the Euclidean metric.
We can write
 $$
\kappa''_{\bar q}(\bar v)=\frac{1}{|\bar v |^2}\nu''_{\bar q}\cdot \ddot \alpha(0)
$$
where $\alpha=(\beta,\alpha_{n})$ is a regular curve in $U'$ projecting onto $\beta$.
From
$$
\bar X_k=X_k-( X_k\cdot e_{n}) e_{n}\,,
$$
and the definition of $\nu''$ \eqref{defnu2} we have
$$
\kappa''_{\bar q}(\bar v)=\frac{(*(X_1( q)\wedge \dots \wedge X_{n-2}( q)\wedge e_{n}))\cdot \ddot\alpha(0)}{|\dot \beta|^2 | X_1(\alpha)\wedge \dots \wedge X_{n-2}(\alpha)\wedge e_{n}|}\,.
$$
We may assume that $\{X_1(q),\dots,X_{n-2}(q)\}$ is an orthonormal basis of $T_{q}U'$ with respect to the Euclidean metric.  Therefore $\{X_1(q),\dots,X_{n-2}(q),\nu'_{q}, q/R\}$ is an Euclidean orthonormal basis of $\RR^{n}$ and we can split
$\RR^{n}$ in
\begin{equation}\label{split}
\RR^{n}=T_{q}U''\oplus \langle \nu'_q\rangle\oplus \langle q/R \rangle.
\end{equation}
and $e_{n}$ splits accordingly into
$$
e_{n}=e_{n}'+e_{n}''+e_{n}'''\,.
$$
Therefore
$$
*(X_1( q)\wedge \dots \wedge X_{n-2}( q)\wedge e_{n})\cdot\ddot \alpha(0)=*(X_1(q)\wedge \dots \wedge X_{n-2}( q)\wedge e_{n}''')\cdot \ddot \alpha(0)\,,
$$
i.e.
$$
*(X_1( q)\wedge \dots \wedge X_{n-2}( q)\wedge e_{n})\cdot\ddot \alpha(0)=\frac{q_{n}}{R}*\left(X_1( q)\wedge \dots \wedge X_{n-2}( q)\wedge \frac{q}{R}\right)\cdot \ddot \alpha(0)\,.
$$
Since
$$
\nu_{q}'=*\left(X_1( q)\wedge \dots \wedge X_{n-2}(q)\wedge \frac{q}{R}\right)
$$
we obtain
$$
\kappa''_{\bar q}(\bar v)= \frac{q_{n}}{R|\dot \beta(0)|^2}\,\,\frac{\nu'_q\cdot\ddot\alpha(0)}{ |X_1( q)\wedge \dots \wedge X_{n-2}(q)\wedge e_{n}|}\,.
$$
We may assume that $\alpha$ is parametrised by arc length with respect to the hyperbolic metric, i.e.
$$
|\dot \alpha|^2=\alpha_{n}^2
$$
and so
$$
|\dot \beta|^2=\alpha_{n}^2-\dot {\alpha} _{n}^2\,,
$$
which implies
$$
\kappa''_{\bar q}(\bar v)= \frac{q_{n}}{r(q_{n}^2-v_{n}^2)}\,\,\frac{\nu'_q\cdot\ddot\alpha(0)}{ |X_1( q)\wedge \dots \wedge X_{n-2}(q)\wedge e_{n}|}\,.
$$
Finally
$$
X_1(q)\wedge \dots \wedge X_{n-1}(q)\wedge e_{n+1}=X_1( q)\wedge \dots \wedge X_{n-1}(q)\wedge e''_{n+1}+X_1( q)\wedge \dots \wedge X_{n-1}(q)\wedge e_{n+1}'''
$$
and
\begin{eqnarray*}
&& X_1( q)\wedge \dots \wedge X_{n-2}(q)\wedge e''_{n}=(\nu_q'\cdot e_{n})\,X_1( q)\wedge \dots \wedge X_{n-2}(q)\wedge\nu_q'\,,\\
&& X_1( q)\wedge \dots \wedge X_{n-2}(q)\wedge e'''_{n}=\frac{q_{n}}{R}X_1(q)\wedge \dots \wedge X_{n-2}(q)\wedge \frac{q}{R}\,.
\end{eqnarray*}
Hence
$$
|X_1(q)\wedge \dots \wedge X_{n-2}(q)\wedge e_{n}|=\left((\nu_q'\cdot e_{n})^2+\frac{q_{n}^2}{R^2}\right)^{1/2} \,.
$$
Now we set 
$$
\kappa_q'(v)=g_q(N'_q,D_t\dot \alpha_{|t=0})
$$
where $D_t$ is the covariant derivative in $\pi$. We have 
$$
D_t\dot \alpha=\ddot \alpha+\sum_{i,j,k=1}^{n} \Gamma_{ij}^k(\alpha)\dot \alpha_i\dot \alpha_j\,e_k=\ddot \alpha+\sum_{i=1}^{n}\left(-\frac{2}{\alpha_n}\dot \alpha_i\dot \alpha_n\right)e_i+\frac{1}{\alpha_n}\left(\sum_{i=1}^n\dot \alpha_i^2-\dot \alpha_n^2\right)e_n
$$
and 
$$
D_t\dot \alpha_{|t_=0}=\ddot \alpha(0)-2\frac{v_n}{q_n}\,v+\frac{1}{q_n}\left(q_n^2-v_n^2\right)e_n\,. 
$$
Therefore 
$$
\kappa_q'(v)=g_q\left(N'_q,\ddot \alpha(0)-2\frac{v_n}{q_n}\,v+\frac{1}{q_n}\left(q_n^2-v_n^2\right)e_n\right)=
\frac{1}{q_n}\nu_q'\cdot \ddot\alpha(0)-2\frac{v_n}{q_n^2}\nu_q'\cdot v+\frac{q_n^2-v_n^2}{q_n^2}\nu'_q\cdot e_n\,.
$$
and from 
$$
\nu_q'\cdot \ddot\alpha(0)=q_n \kappa_q'(v)+2\frac{v_n}{q_n}\nu_q'\cdot v-\frac{q_n^2-v_n^2}{q_n}\nu'_q\cdot e_n
$$
we get
$$
\kappa''_{\bar q}(\bar v)= \frac{q_{n}}{R(q_{n}^2-v_{n}^2)} \left((\nu_q'\cdot e_{n})^2+\frac{q_{n}^2}{R^2}\right)^{-1/2}\,\left(q_n \kappa_q'(v)+2\frac{v_n}{q_n}\nu_q'\cdot v-\frac{q_n^2-v_n^2}{q_n}\nu'_q\cdot e_n
\right)
$$
%
for every $v\in T_qU'$, $g_q(v,v)=1$. Therefore
\begin{eqnarray*}
&& \kappa''_1(\bar q)=\frac{q_{n}^2}{R}\left((\nu_q'\cdot e_{n})^2+\frac{q_{n}^2}{R^2}\right)^{-1/2}
\inf_{v\in\mathbb S^{n-2}_q} A_q(v)\,,\\
&&\kappa''_{n-2}(\bar q)=\frac{q_{n}^2}{R}\left((\nu_q'\cdot e_{n})^2+\frac{q_{n}^2}{R^2}\right)^{-1/2}
\sup_{v\in\mathbb S^{n-2}_q}A_q(v)\,,
\end{eqnarray*}
where 
$$
A_{q}(v)=\frac{1}{(q_{n}^2-v_{n}^2)}\,\left(\kappa_q'(v)+2\frac{v_n}{q_n^2}\nu_q'\cdot v-\frac{q_n^2-v_n^2}{q_n^2}\nu'_q\cdot e_n
\right)
$$
and $\mathbb{S}^{n-2}_q=\{v\in T_qU'\,\,:\,\, |v|_q=1\}$. 
Since $|\kappa_i''(\bar q)| \leq \max \{ |\kappa_1'' (\bar q)|, |\kappa_{n-2}'' (\bar q)| \}$, $i=1,\ldots,n-2$, we obtain
\begin{equation} \label{lasus}
|\kappa_i''(\bar q)| \leq \frac{q_{n}^2}{R}\left((\nu_q'\cdot e_{n})^2+\frac{q_{n}^2}{R^2}\right)^{-1/2}
\sup_{v\in\mathbb S^{n-2}_q}|A_q(v)|\,.
\end{equation}
We have
$$
\begin{aligned}
|A_q(v)|\leq &\,
\frac{1}{|q^2_n-v^2_n|}\left(\left|\kappa_q'(v)\right|+2\frac{v_n}{q_n}+\frac{q_n^2-v_n^2}{q_n^2}\right)\\
\leq&\, \frac{1}{|q^2_n-v^2_n|}\left(\left|\kappa_q'(v)\right|+3\right)\,,
\end{aligned}
$$ 
where we have used $q_n^2-v_n^2> 0$, since $|v|_q=1$.  
Since $\RR^{n}=T_qU'  \oplus  \langle \nu_q' \rangle \oplus \langle q/R \rangle $, we have that
$$
q_{n}^2 - v_{n}^2 \geq \left[ \left(\frac{q_{n}}{R}\right)^2+ (\nu_q' \cdot e_{n})^2 \right] q_{n}^2 \,,
$$
and then from \eqref{lasus} we find
\begin{equation} \label{lasus_II}
|\kappa_i''(\bar q)| \leq  \frac{1}{R}\left((\nu_q'\cdot e_{n})^2+\dfrac{q_{n}^2}{R^2}\right)^{-3/2}
\left(\sup_{v\in\mathbb S^{n-2}_q} |\kappa'_q(v)|+3\right) \,,
\end{equation}
which implies \eqref{bound curvatures II}.
\end{proof}

\begin{remark}{\em We will use the previous proposition in the following way: if there exist a constant $c$ such that   $\nu'_q\cdot e_n\geq c$, then \eqref{bound curvatures II} implies }
$$
|\kappa_i''(\bar q)| \leq \frac{1}{c^3R}
\max\{|\kappa_1'(q)|,|\kappa_{n-2}'(q)|+2\}\,,\  i=1,\dots,n-2\,. 
$$
\end{remark}

\section{Proof of theorem \ref{thm approx symmetry 1 direction}}
The set-up is the following: let $S=\partial \Omega$ be a $C^2$-regular closed hypersurface embedded in $\mathbb H^n$, where $\Omega$ is a bounded open set. We assume that $S$ satisfies a uniform touching ball condition of radius $\rho>0$. 

Let $\pi:=\{p_1=0\}$ be the critical hyperplane in the method of moving  planes along the direction $e_1$ and let $S_-=S\cap \{p_1\leq 0\}$ and $S_+^\pi$ be the reflection of $S_+=S\cap \{p_1\geq 0\}$ about $\pi$. From the method of moving planes we have that
$S_+^\pi$ is contained in $\Omega$ and tangent to $S_-$ at a point $p_0$ (internally or at the boundary). 
Let $\Sigma$ and $\hat \Sigma$ be the connected 
component of $S_+^\pi$ and $S_-$ containing $p_0$, respectively.

\subsection{ Preliminary lemmas}
Before giving the proof of Theorem \ref{thm approx symmetry 1 direction}, we need some preliminary results
about the geometry of $\Sigma$. 

For $t> 0$ we set
$$
\Sigma_t = \{p \in \Sigma:\ d_\Sigma(p,\pa \Sigma) \geq  t \}\,.
$$
The following three lemmas show quantitatively that $\Sigma_t$ is connected for $t$ small enough.

\begin{lemma} \label{lemma connected}
Assume 
\begin{equation}\label{trasversale}
\nu_p \cdot e_1 \leq \mu
\end{equation}
for every $p$ on the boundary of $\Sigma$, for some $\mu\leq 1/2$, and let $t_0=\rho\sqrt{1-\mu^2}$. 
Then $\Sigma_t$ is connected for any $0 < t < t_0$. 
\end{lemma}
\begin{proof} 
Let ${\rm pr}\colon \Sigma \to \pi$ be the projection from $\Sigma$ onto $\pi$. Given $p\in \Sigma$, ${\rm pr} (p)$ is defined as the closest point in $\pi$  to $p$.  
Then the boundary of ${\rm pr}(\Sigma)$ in $\pi$ coincides with the boundary $\partial \Sigma$ of $\Sigma$ in $S$. Proposition \ref{prop Luigi I} implies
$$
|\kappa'_i(p)|\leq \frac{1}{ \sqrt{1-(\nu_p\cdot e_1 )^2}}\max \{|\kappa_1(p)|,|\kappa_{n-1}(p)|\} \,,
$$
for any $p\in \partial \Sigma$ and 
$i=1,\ldots,n-1$, where $\kappa'_i$ are the principal curvatures of $\partial \Sigma$ viewed as a hypersurface of $\pi$. The touching ball condition on $S$ yields
\begin{equation} \label{pulizie}
|\kappa'_i(p)|\leq \frac{1}{ \rho \sqrt{1-(\nu_p\cdot e_1 )^2}} \,,
\end{equation}
for $i=1,\ldots,n-1$. Since any point $p \in \partial \Sigma$ satisfies a touching ball condition of radius $\rho$ (considered as a point of $S$), the transversality condition \eqref{trasversale} and \eqref{pulizie} imply that ${\rm pr}(\Sigma)$ enjoys a touching ball condition of radius $\rho'\geq \rho \sqrt{1-(\nu_p\cdot e_1 )^2} \geq t_0$. Therefore if $s<t_0$,
$$
\mathcal{C}_s = \{ z \in \pi :\ d(z, \partial \Sigma) < s \}\,,
$$
is a collar neighborhood of $\partial \Sigma$ in ${\rm pr}(\Sigma)$ of radius $s$. Since $\pi$ is a critical hyperplane in the method of the moving planes, if $p$ belongs to the maximal cap $S_{+}$ then any point on 
the geodesic path connecting $p$ to its projection onto $\pi$ is contained in the closure of $\Omega$.
It follows that ${\rm pr}^{-1}(\mathcal{C}_s)$ contains a collar neighborhood of $\partial \Sigma$ of radius $s$ in $\Sigma$ and, for $t\leq s$, $\Sigma$ can be retracted in $\Sigma_t$ and the claim follows.   
\end{proof}

\begin{lemma} \label{lemma connected II}
There exists $\delta>0$ depending only on $\rho$ with the following property. Assume that there exists a connected component $\Gamma_\de$ of $\Sigma_\de$ such that 
\begin{equation}\label{era da citare}
0 \leq \nu_q\cdot e_1 \leq \frac{1}{8},
\end{equation}
for any $q \in \partial \Gamma_\de$. Then $\Sigma_\de$ is connected.
\end{lemma}

\begin{proof}
Let $\delta \leq \delta_0(\rho)$, where $\delta_0$ is the bound appearing in Proposition \ref{fejahyp2}. In view of \eqref{era da citare}, we can choose a smaller $\de$ (in terms of $\rho$) such that the interior and exterior touching balls at an arbitrary  $q$ in  $\partial\Gamma_\de$ intersect $\pi$, which implies that $ \Sigma \setminus \Gamma_\de$ is enclosed by $\pi$ and the set obtained as the union of all the exterior and interior touching balls to $S^\pi$ (recall that $\Sigma$ is a subset of the reflection $S^\pi$ of $S$ about $\pi$) of radius $\rho$ at the points on $\Gamma_\de$. Since $\delta$ is choosen small in terms of $\rho$, this implies that for any $p \in \Sigma \setminus \Gamma_\de$ there exists $q \in \partial \Gamma_\de$ such that $d_{\Sigma}(p,q) \leq \de$,
and from \eqref{bound on nu N+1} we have that
$$
|N_p - \tau_q^p(N_q)|_p \leq C \de\,,\mbox{ and } g_p(N_p,\tau_q^p(N_q))\geq \sqrt{1-C^2\delta^2}\,,
$$ 
where $C$ depends on $\rho$. Therefore 
$$
\nu_p \cdot e_1=g_p(N_p,\omega_p)\leq g_p(N_p-\tau_q^p(N_q),\omega_q)+g_p(\tau_q^p(N_q),\omega_p) \leq C\delta+g_p(\tau_q^p(N_q),\omega_p)
$$
and by using 
$$
g_p(\tau_q^p(N_q),\omega_p)=g_q(N_q,\tau_p^q(\omega_p))
$$
and triangular inequality, we get 
$$
\nu_p \cdot e_1\leq C\delta+g_q(N_q,\omega_q)+g_q(N_q,\tau_p^q(\omega_p)-\omega_q)\leq C\delta+\nu_q\cdot e_1+|\tau_p^q(\omega_p)-\omega_q|_q\,. 
$$
In particular, the last bound holds for every $p \in \partial \Sigma$. From proposition \ref{STP1} and by choosing $\delta$ small enough in terms of $\rho$, we obtain  
\begin{equation*}
\nu_p \cdot e_1\leq \frac14
\end{equation*}
and lemma \ref{lemma connected} implies the statement. 
\end{proof}

\begin{lemma} \label{lemma connected III}
There exists $\delta>0$ depending only on $\rho$ with the following property. Assume that there exists a connected component $\Gamma_\delta$ of $\Sigma_{\delta}$ such that for any $q\in \partial \Gamma_\delta$ there exists $\hat q \in \hat \Sigma$ such that 
$$
d(q,\hat q)+|N_q-\tau_{\hat q}^q(N_{\hat q})|_q \leq \delta\,,
$$
then 
\begin{equation}  \label{bellachegira}
0\leq \nu_z\cdot e_1\leq \frac14
\end{equation}
for any $z\in \partial \Sigma$ and $\Sigma_\delta$ is connected. 
\end{lemma}
\begin{proof} 
Let $q\in \partial \Gamma_\delta$. 
By construction  $\nu_q\cdot e_1\geq 0$. Let $q^\pi$ be the reflection of $q$ with respect to $\pi$.  By our assumptions we have
$$
d(q^\pi,\hat q) \leq d(q^\pi,q) + d(q,\hat q) \leq 3 \delta \,.
$$
We can choose $\delta$ small enough in terms of $\rho$ and find $C=C(\rho)$ such that: $d_S(q^\pi,\hat q) \leq C \delta $ (as follows from \eqref{bound on d Ga above}), $q^\pi\in \mathcal{U}_{\rho_1}(\hat q)$, and so that 
$$
g_{\hat q}(N_{\hat q}, \tau_{q^\pi}^{\hat q}(N_{q^\pi}))\geq  \sqrt{1-C^2\delta^2}\\ \quad \textmd{ and } \quad \ \ |N_{\hat q} - \tau_{q^\pi}^{\hat q}(N_{q^\pi})|_{\hat q} \leq C \delta 
$$
(see \eqref{bound on nu N+1}). Since $N_{q^\pi}= (-(N_q)_1, (N_q)_2, \ldots, (N_q)_n)$ and $q$ and $q^\pi$ are symmetric about $\pi$, we have that
$$
\nu_q\cdot e_1=g_q(N_q,\omega_q)=-g_q(\tau_{q^\pi}^{q}(N_{q^\pi}),\omega_q) \,,$$
and so
$$
2g_q(N_q,\omega_q)=g_q(N_q-\tau_{q^\pi}^{\hat q}(N_{q^\pi}),\omega_q)=g_q(N_q-\tau_{\hat q}^{ q}(N_{\hat q}),\omega_q)+g_q(\tau_{\hat q}^{q}(N_{\hat q})-\tau_{q^\pi}^{q}(N_{q^\pi}),\omega_q) \,.
$$
This implies that 
$$
0\leq 2g_q(N_q,\omega_q)\leq |N_q-\tau_{\hat q}^{q}(N_{\hat q})|_q+|\tau_{\hat q}^{q}(N_{\hat q})-\tau_{q^\pi}^{q}(N_{q^\pi})|_q \,,
$$
and lemma \ref{STP2} together with our assumptions implies 
\begin{equation} \label{treno}
0\leq 2\nu_q \cdot e_1\leq\frac18\,.
\end{equation}
From Lemma \ref{lemma connected II} we obtain that $\Sigma_\delta$ is connected.

%
%

Now fix $z \in \partial \Sigma$ and let $q$ be such that $d_{\Sigma} (q,z) = \delta$ (so that  $z \in \mathcal{U}_\rho(q)$). Since $q$ and $q^\pi$ are symmetric about $\pi$, then we have that
$$
g_z(\tau_q^z(N_q),\omega_z) = - g_z(\tau_{q^\pi}^z(N_{q^\pi}),\omega_z) \,,
$$ 
and hence
$$
2 g_z(\tau_q^z(N_q),\omega_z) = g_z(\tau_q^z(N_q) - \tau_{q^\pi}^z(N_{q^\pi}),\omega_z) \,.
$$
We write
$$
\begin{aligned}
2 g_z(N_z,\omega_z) &  = 2g_z(\tau_{q}^z(N_q),\omega_z) + 2g_z(N_z - \tau_{q}^z(N_q),\omega_z)  \\ 
& = g_z(\tau_{q}^z(N_q) -\tau_{q^\pi}^z(N_{q^\pi}),\omega_z) +  2 g_z(N_z - \tau_{q}^z(N_q),\omega_z) \\
& =  g_z(\tau_{q}^z(N_q) - \tau_{\hat q}^z(N_{\hat q}), \omega_z) + g_z(\tau_{\hat q}^z(N_{\hat q}) ,\omega_z) - g_z(\tau_{q^\pi}^z(N_{q^\pi}),\omega_z) +  2 g_z(N_z - \tau_{q}^z(N_q),\omega_z) 
\,.
\end{aligned}
$$
By Cauchy-Schwarz and triangle inequalities we have 
$$
\begin{aligned}
|2 g_z(N_z,\omega_z)| & \leq |\tau_{q}^z(N_q) - \tau_{\hat q}^z(N_{\hat q})|_z + |\tau_{\hat q}^z(N_{\hat q}) - \tau_{q^\pi}^z(N_{q^\pi})|_z +  2 |N_z -\tau_{q}^z(N_q)|_z \\
&  \leq  |\tau_{q}^z(N_q) - \tau_{\hat q}^z(N_{\hat q})|_z + |\tau_{\hat q}^z(N_{\hat q}) - \hat N_z|_z + |\hat N_z- \tau_{q^\pi}^z(N_{q^\pi})|_z +  2 |N_z - \tau_{q}^z(N_q)|_z
\,,
\end{aligned}
$$
where $N_z$ and $\hat N_z$ are the normal vectors to $\Sigma$ and $\hat \Sigma$ at $z$, respectively.  The first term can be bounded in terms of $\delta$ by lemma \ref{STP2}. All the remaining terms on the right hand side can be estimated in terms of $\delta$ by using  proposition \ref{fejahyp2}.
This implies that 
$$
|2 g_z(N_z,\omega_z)| \leq C \delta \,.
$$
By choosing $\delta$ small enough compared to $C$ (and hence compared to $\rho$) we have that 
\begin{equation*}
0 \leq \nu_z \cdot e_1 \leq 1/4,
\end{equation*}
i.e. $\Sigma$ intersects $\pi$ transversally. 
\end{proof}

The following lemma will be used several times in the proof of theorem \ref{thm approx symmetry 1 direction}.


\begin{lemma} \label{lem_p_pstar}
Assume that $e_n \in \Sigma$ with $\nu_{e_n}=e_n$ and that there exist two local parametrizations $u,\hat u: B_{r} \to \RR$ of $\Sigma$ and $\hat \Sigma$, respectively, with $0<r\leq \rho_1$ and such that $u -\hat u  \geq 0$, where $\rho_1$ is given by \eqref{rho1}.

Let $p_1=(x_1,u(x_1))$ and $\hat p_1^*=(x_1,\hat u(x_1))$, with $x_1\in \partial B_{r/4}$, and denote by $\gamma$ the geodesic path starting from $p_1$ and tangent to $\nu_{p_1}$ at $p_1$. Assume that 
\begin{equation} \label{conan}
d(p_1,\hat p_1^*) + |\nu_{p_1} - \nu_{\hat p_1^*} | \leq \theta \,. 
\end{equation}
for some $\theta \in [0, 1/2]$. There exists $\bar r$ depending only on $\rho$ such that if $r \leq \bar r$ we have that $\gamma \cap \hat \Sigma \neq \emptyset$ and, if we denote by $\hat p_1$ the first intersection point between $\gamma$ and $\hat \Sigma$, then
\begin{equation*}
d(p_1,\hat p_1) + |N_{p_1} - \tau_{\hat p_1}^{p_1}(N_{\hat p_1})|_{p_1} \leq C \theta \,,
\end{equation*}
where $C$ is a constant depending only on $n$ and $\rho$, and provided that $C\theta < 1/2$.
\end{lemma}

\begin{proof}
We first notice that, by choosing $r$ small enough in terms of $\rho$, from Lemma \ref{fejahyp} we have that $|\nu_{p_1} - e_n | \leq  1/4$. By using the touching ball condition for $\hat \Sigma$ at $\hat{p}_1^*$, a simple geometrical argument shows that the geodesic passing through $p_1$ and tangent to $\nu_{p_1}$ at $p_1$ intersects $\hat \Sigma$, so that $\hat p_1$ is well defined.   

As shown in figure \ref{figLuigi}, 
we estimate the distance between $p_1$ and $\hat p_1$ as follows. Let $q$ be the unique point having distance $2\ep$ from $p_1$ and lying on the geodesic path containing $p_1$ and $\hat p_1^*$. Let $T$ be the geodesic right-angle triangle having vertices $p_1$ and $q$ and hypotenuse contained in the geodesic passing  through $p_1$ and $\hat p_1$. Since the angle $\alpha$ at the vertex $p_1$ is such that $|\sin \al| \leq 1/4$, then from the sine rule for hyperbolic triangles we have that 
\begin{equation} \label{conad2}
d(p_1, \hat p_1) \leq C \theta \,.
\end{equation}
Moreover, the cosine law formula in hyperbolic space gives that
\begin{equation} \label{conad3}
d(\hat p_1^*, \hat p_1) \leq C \theta 
\end{equation}
for some constant $C$, and from \eqref{bound on nu N+1} we obtain that 
\begin{equation} \label{coop1}
|N_{p_1} - \tau_{\hat p_1}^{p_1}(N_{\hat p_1})|_{p_1} \leq |N_{p_1} - \tau_{\hat p_1^*}^{p_1}(N_{\hat p_1^*})|_{p_1}+ |\tau_{\hat p_1^*}^{p_1}(N_{\hat p_1^*}) - \tau_{\hat p_1}^{p_1}(N_{\hat p_1})|_{p_1} \,.
\end{equation}
Since $p_1$ and $\hat p_1^*$ lie on the same vertical line, we have that
\begin{equation} \label{coop2}
|N_{p_1} - \tau_{\hat p_1^*}^{p_1}(N_{\hat p_1^*})|_{p_1} = |\nu_{p_1} - \nu_{\hat p_1^*}| \leq C \theta \,,
\end{equation}
where the last inequality follows from \eqref{conad}. Moreover, from proposition \ref{STP2} we have
$$
\begin{aligned}
|\tau_{\hat p_1^*}^{p_1}(N_{\hat p_1^*}) - \tau_{\hat p_1}^{p_1}(N_{\hat p_1})|_{p_1} & \leq C \Big(d(p_1,\hat p_1^*) + d(p_1,\hat p_1) + d(\hat p_1, \hat p_1^*) + |N_{\hat p_1} - \tau_{\hat p_1^*}^{\hat p_1} (N_{\hat p_1^*}) |_{\hat p_1} \Big) \\
& \leq C \theta \,,
\end{aligned}
$$
where the last inequality follows from \eqref{conad},\eqref{conad2},\eqref{conad3} and \eqref{bound on nu N+1}. This last inequality, \eqref{coop1} and \eqref{coop2} imply that
$$
|N_{p_1} - \tau_{\hat p_1}^{p_1}(N_{\hat p_1})|_{p_1} \leq C \theta \,,
$$
and therefore from \eqref{conad2} we conclude.
\end{proof}

\begin{figure}[h]
\centering
\begin{tikzpicture} [scale=2]
\draw (2,2).. controls(0,0.25)..(-2,2) node[left] {\footnotesize $\Sigma$};
\draw  (2,1.2).. controls(0,0.27).. (-2,1.2) node[left] {\footnotesize $\hat \Sigma$};
\draw (-2.3,0.2) -- (2,0.2);
\node at (1/7,0.8) {\footnotesize $e_n$};
\node at (2,0.1) {\footnotesize $\pi_{\infty}$};
\draw[->] (0,0.69)--(0,1);
\filldraw [black] (0,0.69) circle (0.5 pt);
\draw [domain=0:70] plot ({0.1+1.7*cos(\x)}, {0.2+1.7*sin(\x)});
\filldraw [black] (1.59,1.01) circle (0.5 pt);
\filldraw [black] (1.31,1.40) circle (0.5 pt)node[left] {\footnotesize $ p_1$};
\draw[-] (1.31,1.40)--(1.31,0.5);
\filldraw [black] (1.31,0.89) circle (0.5 pt);
\filldraw [black] (1.31,0.5) circle (0.5 pt)node[left] {\footnotesize $ q$};
\node at (1.8,0.95) {\footnotesize $\hat p_1$};
\node at (1.2,0.7) {\footnotesize $\hat p_1^*$};
\filldraw [black] (1.24,1.70)node[left] {\footnotesize $ \gamma$};
\end{tikzpicture}
\caption{}\label{figLuigi}
\end{figure}

\subsection{Proof of the first part of theorem \ref{main}. }Now we can focus on the proof of the first part of theorem \ref{thm approx symmetry 1 direction}, showing that there exist constants $\ep$ and $C$, depending only on $n$, $\rho$ and $|S|_g$, such that if
$$
{\rm osc}( H) \leq \ep,
$$
then for any $p$ in $\Sigma$ there exists $\hat p$ in $\hat\Sigma$ satisfying
\begin{equation}\label{bound on dist}
d(p,\hat p) +  |N_p-\tau_{\hat p}^p (N_{\hat p})|_p \leq C\, \oscH  \,. 
\end{equation}

We will have to choose a number $\delta>0$ sufficiently small in terms of $\rho$, $n$ and $|S|_g$ and subdivide the proof of the first part of the statement in four cases depending on the whether the distances of $p_0$ and $p$ from $\partial \Sigma$ are greater or less than $\delta$. A first requirement on $\delta$ is that it satisfies the assumptions of lemmas \ref{lemma connected II} and \ref{lemma connected III}; other restrictions on the value of $\delta$ will be done in the development of the proof.

\subsubsection{Case 1. $d_{\Sigma}(p_0,\pa \Sigma) > \de$ and $d_{\Sigma}(p,\pa \Sigma) \geq \de$} \label{subsec case1}

In this first case we assume that $p_0$ and $p$ are interior points of $\Sigma$, which are far from $\pa \Sigma$ more than $\de$.
We first assume that $p_0$ and $p$ are in the same connected component of $\Sigma_\de$; then, lemma \ref{lemma connected II} will be used in order to show that $\Sigma_\de$ is in fact connected.

From lemma \ref{lemma bound on d Ga} we can choose $\alpha \in (0, \frac{1}{2} \min(1,\rho_1^{-1}))$ such that $\alpha C \rho_1 \leq \delta/4$, where $C$ is the constant appearing in \eqref{disk}, and we set 
\begin{equation} \label{rpippo}
r_0=\min(\bar r,\alpha \rho_1 )\,,
\end{equation}
where $\bar r$ is given by lemma \ref{lem_p_pstar}. Accordingly to this definition of $r_0$, from \eqref{disk} we have that if $p_i\in \Sigma_\delta$ then 
$\mathcal U_{r_0}(p_i)\subset \Sigma$. 
\begin{lemma}\label{catenadipalle}
Let $\ep_0 \in [0,1/2]$, $p_0$ and $p$ be in a connected component of $\Sigma_\delta$ and $r_i=(1-\ep_0^2)^ir_0$.  There exist an
integer $J\leq J_\de$, where 
\begin{equation}\label{N_delta_hyp}
J_\de:=\max\left(4,\frac{2^{n-1} |S|_g}{ \de^{n-1}} \right) \,,
\end{equation}
and a sequence of points $\{p_1,\dots,p_J\}$ in $\Sigma_{\delta/2}$ such that
\begin{eqnarray*}
&& p_0,p \in \bigcup_{i=0}^J \overline{\mathcal U}_{r_i/4}(p_i)\,,\\
&& \mathcal U_{r_0}(p_i)\subseteq \Sigma, \quad i = 0,\ldots, J\,,\\
&& p_{i+1} \in\overline{\mathcal U}_{r_i/4}(p_i), \quad i=0,\ldots,J-1\,.
\end{eqnarray*}
\end{lemma}
\begin{proof}
For every $z$ in $\Sigma$ and $r\leq \rho_0$, the geodesic ball $\mathcal B_{r}(z)$ in $\Sigma$ satisfies 
$$ 
|\mathcal B_{r}(z)|_\Sigma\geq c r^{n-1}
$$
where $c$ is a constant depending only on $n$ (see formula \eqref{geodesic_disk}). A general result for Riemannian manifolds with boundary (see e.g. proposition \ref{prop harnack chain}) implies that there exists a piecewise geodesic path parametrized by arc length $\gamma\colon  [0,L]\to \Sigma_{\delta/2}$ connecting $p_0$ to $p$ and of length $L$ bounded by $\delta J_{\delta}$, where $J_\delta$ is given by \eqref{N_delta_hyp}.

%
We define $p_i=\gamma(r_i/4)$, for $i=1,\dots,J-1$  and $p_J=p$. Our choice of $r_0$ guarantees that $\mathcal U_{r_0}(p_i)\subseteq \Sigma$, for every $i = 0,\ldots, J$, and the other required properties are satisfied by construction.
\end{proof}

Since $p$ and $p_0$ are in a connected component of $\Sigma_\de$, there exist $\{p_1,\dots,p_J\}$ in the connected component of $\Sigma_{\delta/2}$ containing $p_0$ and a chain of subsets $\{\mathcal U_{r_0}(p_i)\}_{\{i=0,\ldots,J\}}$ of $\Sigma$ as in lemma \ref{catenadipalle}.
We notice that $\Sigma$ and  $\hat\Sigma$ are tangent at
$p_0$ and that in particular the two normal vectors to $\Sigma$ and $\hat \Sigma$ at $p_0$ coincide. Up to an isometry we can assume that $p_0=e_{n}$ and $\nu_{p_0}=e_{n}$, and then $\Sigma$ and $\hat\Sigma$ can be locally represented near $p_0$ as the graphs of two functions
$u_0,\, \hat u_0\colon B_{r_0} \subset \pi_\infty \to \RR$.
Lemma \ref{fejahyp} implies that $| \nabla u_0|, |\nabla \hat u_0 | \leq M$ in $B_{r_0}$, where $M$ is some constant which depends only on $r_0$, i.e. only on $\rho$.  Since $u_0$ and $\hat u_0$ satisfy \eqref{Hu} and $| \nabla u_0|, |\nabla \hat u_0 | \leq M$, then the difference $u_0-\hat u_0$ solves a second-order linear uniformly elliptic equation of the form
$$
\mathcal L(u_0-\hat u_0)(x)=H(x,u(x))-\hat H(x,\hat u(x))
$$
with ellipticity constants uniformly bounded by a constant depending only on $n$ and $\rho$. Since  $u_0(0)=\hat u_0(0)$ and $u_0 \geq \hat u_0$,
Harnack inequality (see Theorems 8.17 and 8.18 in \cite{GT}) yields  
$$
\sup_{B_{r_0/2}} (u_0-\hat u_0)\leq C\,{\rm osc} (H) \,,
$$
and from interior regularity estimates (see e.g. \cite[Theorem 8.32]{GT}) we obtain 
\begin{equation}\label{harnack step 1}
\|u_0- \hat u_0\|_{C^1(B_{r_0/4})} \leq C \,{\rm osc}( H),
\end{equation}
where $C$ depends only on $\rho$ and $n$. 

Since $p_1\in \partial\, \mathcal U_{r_0/4}(p_0)$, we can write $p_1=(x_1,u_0(x_1))$, with $x_1\in \partial B_{r_0/4}$, and define $\hat p_1^*$ and $\hat p_1$ as in lemma \ref{lem_p_pstar}. We notice that \eqref{harnack step 1} yields
\begin{equation} \label{conad}
d(p_1,\hat p_1^*) + |\nu_{p_1} - \nu_{\hat p_1^*} | \leq C\oscH \,, 
\end{equation}
so that \eqref{conan} in lemma \ref{lem_p_pstar} is fullfilled. From  lemma \ref{lem_p_pstar} we find
\begin{equation} \label{esselunga}
d(p_1,\hat p_1) + |N_{p_1} - \tau_{\hat p_1}^{p_1}(N_{\hat p_1})|_{p_1} \leq C \oscH \,.
\end{equation}


Now we apply an isometry in such a way that $p_1=e_n$ and $\nu_{p_1} = e_n$. We notice that by construction $\hat p_1$ becomes of the form $\hat p_1=te_n$, with $t\geq 1$ (notice that $t=1+d(p_1,\hat p_1)$).
From the Euclidean point of view, in this configuration $\mathcal{U}_{r_0}(p_1) \subset \Sigma$ satisfies an Euclidean touching ball condition of radius $\rho_1$. Moreover, being $\hat p_1=te_n$ with $t\geq 1$ also $\hat{\mathcal{U}}_{r_0}(p_1) \subset\hat \Sigma$ satisfies the Euclidean touching ball condition of radius $\rho_1$. 
Since in this configuration we have that 
$$
|\nu_{p_1}-\nu_{\hat p_1}|= |N_{p_1} - \tau_{\hat p_1}^{p_1}(N_{\hat p_1})|_{p_1} \,,
$$ 
from \eqref{esselunga} we find
$$
|\nu_{p_1}-\nu_{\hat p_1}| \leq C \oscH \,,
$$
where $C$ is a constant that depends only on $\rho$ and $n$. A suitable choice of $\ep$ in the statement of theorem \ref{main} (i.e. such that  $C\ep < 1$) guarantees that we can apply lemma \ref{lemma change normal I} (recall that $\oscH \leq \ep$) and we obtain that $\Sigma$ and $\hat \Sigma$ are locally graphs of two functions
$$
u_1,\hat u_1: B_{r_1} \to \RR^+ \,,
$$
such that $u_1(0)=p_1$ and $\hat u_1(0)= \hat p_1$ and where
$$
r_1=(1-\ep^2)r \,.
$$ 
Now, we can iterate the argument before. Indeed, since 
$$
0 \leq \inf_{B_{r_1/2}} (u_1-\hat u_1) \leq u_1(0)-\hat u_1(0) \leq C \oscH\,,
$$
by applying Harnack's inequality we obtain that
$$
\sup_{B_{r_1/2}} (u_1-\hat u_1)\leq C\,{\rm osc} (H)
$$
and from interior regularity estimates we find 
\begin{equation}\label{harnack step 2}
\|u_1- \hat u_1\|_{C^1(B_{r_1/4})} \leq C \,{\rm osc}( H),
\end{equation}
where $C$ depends only on $\rho$ and $n$. Hence, \eqref{harnack step 2} is the analogue of \eqref{harnack step 1}, and we can iterate the argument. The iteration goes on until we arrive at $p_N=p$ and obtain a point $\hat p_N \in \hat \Sigma$ such that
$$
d(p,\hat p_N) + |N_p-\tau_{\hat p_N}^{p}(N_{\hat p_N})|_p \leq C \oscH \,.
$$
In view of lemma \ref{lemma connected III} we have that $\Sigma_\delta$ is connected and the claim follows.

\medskip
\subsubsection{Case 2: $d_\Sigma(p_0,\pa \Sigma) \geq \de$ and $d_\Sigma(p,\pa \Sigma) < \de$} \label{subsec case2}
Here we extend the estimates found at case 1 to a point $p$ which is far less than $\delta$ from the boundary of $\Sigma$. Let $q\in \Sigma$ and $p_{min}\in \partial \Sigma$ be such that
$$
d_{\Sigma}(q,\partial \Sigma)=\delta \,, \quad d_{\Sigma}(p,q)+d_{\Sigma}(p,\partial \Sigma )=\delta\,, \quad  \mbox{and } d_{\Sigma}(p,p_{min})=d_{\Sigma}(p,\partial \Sigma)\,.
$$
From case 1 we have that there exists $\hat q$ in $\hat \Sigma$ such that
$$
d(q,\hat q)+|N_q-\tau_{\hat q}^q(N_{\hat q})|_q \leq C\,{\rm osc} (H)\,.
$$
Lemma \ref{lemma connected III} yields that 
\begin{equation} \label{bellachegira}
0 \leq g_z(N_z,\omega_z) \leq 1/4,
\end{equation}
for any $z \in \partial \Sigma$ and $\Sigma_\delta$ is connected.

For $r\leq \rho_1$, with $\rho_1$ given by \eqref{rho1}, we define $U_{r}(q)$ as the reflection of
$\mathcal U_r(q^\pi)\cap \{x_1\geq 0\}$ with respect to $\pi$ and $U'=\mathcal{U}_r(q^\pi)\cap \{x_1= 0\}$.  From proposition \ref{prop Luigi I}, $U'$ is a hypersurface of $\pi$ with a natural orientation and its principal curvatures $\kappa_i'$ satisfy
\begin{equation*}
\frac{1}{\sqrt{1-g_z(N_z,\omega_z)^2}}\kappa_1(z)\leq \kappa'_i(z)\leq \frac{1}{ \sqrt{1-g_z(N_z,\omega_z)^2} }\kappa_{n-1}(z) \,,
\end{equation*}
for every $z\in U'$ and $i=1,\dots,n-1$. From \eqref{bellachegira} and since $|\kappa_i(z)| \leq \rho^{-1}$ for any $z \in S$ (this follows from the touching sphere condition), we have
\begin{equation}\label{nu}
|\kappa'_i(z)|\leq \frac{2}{\rho} \,.
\end{equation}

Now we apply an isometry  $f\colon \mathbb H^n\to\mathbb H^n $ such that $f(q)=e_{n} $ and the normal vector to $f(S)$ at $f(q)$ is $e_{n}$ (i.e. $f_{*|q}(N_q)=e_n$). 

Let $U''$ be the Euclidean orthogonal projection of $f(U')$ onto $\pi_{\infty}$. Our goal is to estimate the curvatures of $U''$. It is clear that $f(\pi)$ is either a vertical hyperplane or a half-sphere intersecting $f(S)$. In the first case we immediately conclude since the curvatures of $U''$ vanish. 

Thus, we assume that $f(\pi)$ is a half-sphere. A straightforward computation yields that the radius of $f(\pi)$ is
$$
R=\frac{q_{n}(\Theta^2+1)}{2|\Theta|\,|a\Theta+q_{n} |}\,,
$$
where
$$
\Theta=-\frac{\sin \theta}{1+\cos \theta}\,,\quad \cos \theta =\nu_q\cdot e_{n}
$$
and $a$ is the Euclidean distance of $q$ from $\pi$.  It is easy to see that
$$
a\leq q_{n}\sinh(\delta)
$$
and so
$$
\frac{1}{R} \leq \frac{2|\Theta|\,(\sinh(\delta)|\Theta|+1 )}{\Theta^2+1} 
$$
which implies 
\begin{equation}\label{r}
\frac1R \leq 1 + 2 \sinh (\de) \,.
\end{equation}
We notice that the last estimate can be alternative found by noticing that an isometry that fixes $e_n$ maps a vertical hyperplane into a half sphere, where the radius can be estimated by using the distance of $e_n$ from the vertical hyperplane.

We still denote by $\nu'$ the Euclidean normal vector field to $f(U')$. We denote by $\kappa''_i$ the principal curvatures of $U''$ with respect to the Euclidean metric on $\pi_{\infty}$ and a chosen orientation.
Now, we want to find an upper bound on the curvatures of $U''$ which will allow us to use Carleson type estimates. Proposition \ref{prop Luigi II} and formula  \eqref{r} imply
$$
\begin{aligned}
|\kappa_i''(\bar \xi)| &  \leq \frac{1}{R}
 \left((\nu_\xi'\cdot e_{n})^2+\dfrac{\xi_{n}^2}{R^2}\right)^{-3/2} \left(2+
\max\{|\kappa_1'(f^{-1}(\xi))|,|\kappa_{n-1}'(f^{-1}(\xi))|\} \right) \\
 & \leq \frac{1 + 2 \sinh \de}{ |\nu_\xi'\cdot e_{n}|^3}\left(2+
\max\{|\kappa_1'(f^{-1}(\xi))|,|\kappa_{n-1}'(f^{-1}(\xi))|\} \right)
\end{aligned}
$$
for every $\xi=(\bar \xi,\xi_{n})$ in $f(U')$ and $i=1,\dots,n-2$. 
Then \eqref{nu} yields that 
\begin{equation}\label{sbrillano}
|\kappa_i''(\bar \xi)| \leq \frac{2(1+\rho)(1+2\sinh \de)}{ \rho|\nu_\xi'\cdot e_{n}|^3}\,.
 \end{equation}

Next we show 
\begin{equation} \label{nuxinuxip}
\nu_\xi'\cdot e_{n} \geq 1/2 \,.
\end{equation} 
We write
$$
\nu_\xi'\cdot e_{n} = \nu_\xi'\cdot (e_{n} - \nu_\xi) + \nu'_\xi \cdot \nu_\xi \,,
$$
where we still denote by $\nu$ the normal vector field to $f(S)$.  
Since $f_{*|q}(\nu_q)=e_n$, from lemma 2.1 in \cite{JEMS} we have that $|e_{n} - \nu_\xi| \leq 1/4$ by choosing $r$ small enough in terms of $\rho_1$ and hence of $\rho$. Moreover, since 
$$
\nu'_\xi \cdot \nu_\xi =\nu'_{f^{-1}(\xi)} \cdot \nu_{f^{-1}(\xi)}\,,
$$
\cite[formula (2.29)]{JEMS} implies 
$$
\nu'_{f^{-1}(\xi)} \cdot \nu_{f^{-1}(\xi)} = \sqrt{1-(\nu_{f^{-1}(\xi)} \cdot e_1)^2} 
$$
and \eqref{bellachegira} gives \eqref{nuxinuxip}.
%
Therefore 
\begin{equation}\label{lucionediavolo}
|\kappa_i''(\bar \xi)| \leq C\,.
\end{equation}
for some constant $C=C(\rho)$. 

Let $x=\overline{f(p_{min})}$ and $y=\overline{f(p)}$ be the projections of $f(p_{min})$ and $f(p)$ onto $\pi_\infty$, respectively, and let $E_r$ be the projection of $f(U_r(q))$ onto $\pi_\infty$. From \eqref{dist_equiv} we have that $|x-y| \leq C\de$, with $C\geq 1$ which depends only on $\rho$.  We can choose $\delta$ small enough (compared to $\rho$) such that $B_{8 C \delta}(x) \cap \partial E_r \subset U''$, apply theorem 1.3 in \cite{BCN} and corollary 8.36 in \cite{GT} and find
\begin{equation} \label{violacarleson}
\sup_{B_{2C\de}(x) \cap E_r} (u - \hat u) \leq C_1 (u - \hat u)(z) + \oscH \,,
\end{equation}
with $z=x + 4C \delta \nu_x''$, where $\nu_x''$ is the interior normal to $U''$ at $x$. By choosing $\delta$ small enough in terms of $\rho$, the bound on the curvatures of $U''$ implies that the point $z$ has distance $4C\de$ from the boundary of $E_r$. Since $d_\Sigma(q,U')= \delta$, then the distance (in $\pi_\infty$) of $O$ from the boundary of $E_r$ is at least $c\delta$ (as follows from \eqref{dist_equiv}), where $c<C$ depends only on $\rho$. From Harnack's inequality 
$$
C_1 (u - \hat u)(z) + \oscH \leq C_2(u(0) - \hat u (0)  + \oscH) \,,
$$
and from \eqref{violacarleson} we obtain that 
\begin{equation*}
0 \leq \sup_{B_{2C\de}(x) \cap E_r} (u - \hat u) \leq C_2(u(0) - \hat u (0)  + \oscH) \,.
\end{equation*}
Boundary regularity estimates (see e.g. \cite[Corollary 8.36]{GT}) yield
\begin{equation} \label{violaharnack}
0 \leq  \|u - \hat u\|_{C^1(B_{C\de}(x) \cap E_r)} \leq C_3 ( (u(0) - \hat u (0))  + \oscH )\,.
\end{equation}
Since $d_\Sigma(q, \partial \Sigma)=\delta$, from Case 1 we know that 
\begin{equation*}
d(q,\hat q) + |N_{q} - \tau_{\hat q}^q(N_{\hat q}) |_q \leq C\oscH \,,
\end{equation*}
where $\hat q$ is the first intersecting point between $\hat \Sigma$ and the geodesic path starting form $q$ and tangent to $\nu_{q}$ at $q$ (recall that $f(q)=e_n$ and $N_q=e_n$). From \eqref{violaharnack} we obtain that 
\begin{equation} \label{violaharnack_I}
0 \leq  \|u - \hat u\|_{C^1(B_{C\de}(x) \cap E_r)} \leq C \oscH \,.
\end{equation}
We define $\hat p^*$ so that $\hat p^* = f(y,\hat u(y))$. Since $y \in B_{C\de}(x)$, \eqref{violaharnack_I} implies 
$$
d(f(p),f(\hat p^*)) + |\nu_{f(p)} - \nu_{f(\hat p^*)}| \leq C \oscH \,.
$$
Since $f(p)$ and $f(\hat p^*)$ are on the same vertical line, we can write 
$$
d(f(p),f(\hat p^*)) + |N_{f(p)} - \tau (N_{f(\hat p^*)}) |_{f(p)} \leq C \oscH \,,
$$
where $\tau$ is the parallel transport along the vertical segment connecting $f(\hat p^*)$ with $f(p)$. Lemma \ref{lem_p_pstar} yields 
$$
d(p,\hat p) + |N_{p} - \tau_{\hat p}^p (N_{\hat p}) |_{p} \leq C \oscH \,,
$$ 
as required.

\subsubsection{Case 3: $0 < d_\Sigma(p_0,\pa \Sigma) < \de$.}

We first show that the center of the interior touching ball of radius $\rho $ to $S$ at $p_0$, say $\mathsf{B}_{\rho}(a)$, lies on the left of $\pi$, i.e. $a \cdot e_1 \leq 0$.  Indeed, since $p_0$ is  a tangency point, $p_0^\pi\in S$  and hence $p_0^\pi$ does not lie in $\mathsf{B}_{\rho}(a)$. This implies
$$
d(p_0,a)=\rho\leq d(p_0^\pi,a) \,,
$$
and since $p_0$ and $p_0^\pi$ have the same height we have 
$$
|p_0-a|^2\leq |p_0^\pi-a|^2 \,,
$$
which implies that $a\cdot e_1\leq 0$.


Now we prove that $\Sigma$ and $\pi$ intersect transversally at $p_0$ (see \eqref{quasi_ortog_p0} below). Since $d(p_0,\pi) \leq d_{\Sigma}(p_0, \pa \Sigma) \leq \de$, then $d(p_0,p_0^\pi) \leq 2 \de$. We can choose $\delta$ small in terms of $\rho$ so that $p_0^\pi \in \mathcal{U}_{\rho_1}(p_0)$. From \eqref{bound on nu N+1} we have that 
\begin{equation} \label{uuuhh}
g_{p_0}(N_{p_0}, \tau_{p_0^\pi}^{p^\pi}(N_{p_0^\pi}))\geq  \sqrt{1-C^2\delta^2}\\ \quad \textmd{ and } \quad \ \ |N_{p_0} - \tau_{p_0^\pi}^{p_0}(N_{p_0^\pi})|_{p_0} \leq C \delta\,.
\end{equation}
Since 
$$
g_{p_0}(N_{p_0}, \omega_{p_0}) = - g_{p_0^\pi}(N_{p_0}, \omega_{p_0^\pi}) \,,
$$
and $g_{p_0}(N_{p_0}, \omega_{p_0}) \geq 0$ by construction, then
$$
0 \leq 2 g_{p_0}(N_{p_0}, \omega_{p_0}) = g_{p_0}(N_{p_0} - \tau_{p_0^\pi}^{p_0}(N_{p_0^\pi}), \omega_{p_0}) \leq | N_{p_0} - \tau_{p_0^\pi}^{p_0}(N_{p_0^\pi}) |_{p_0} \leq C\delta \,,
$$
where the last inequality follows from \eqref{uuuhh}. By choosing $\delta$ small compared to $C$ (in terms of $\rho$) we have 
\begin{equation} \label{quasi_ortog_p0}
0 \leq g_{p_0}(N_{p_0}, \omega_{p_0}) \leq \frac{1}{4} \,.
\end{equation}

Now we apply an isometry $f\colon \mathbb H^n \to \mathbb H^n$ such that  $f(p_0)=e_n$ and $f_{*|p_0}(N_{p_0})=e_n$. 
As for Case 2 (with $q$ replaced by $p_0$), we locally write $f(\Sigma)$ and $f(\hat \Sigma)$ as graphs of function $u, \hat u\colon E_r\to \RR$, respectively. Moreover, we denote by $U''$ the portion of $\pa E_r$ which is obtained by projecting $f({U}_r(p_0) \cap \pi)$ onto $\pi_{\infty}$. We remark that $u=\hat u$ on $U''$ and, again by arguing as in Case 2, that the principal curvatures of $U''$ can be bounded by a constant $\mathcal{K}$ depending only on $\rho$.

Let $\bar x \in U''$ be a point such that
$$|\bar x| = \min_{x \in U''} |x|.$$
Notice that $|\bar x| \leq C d_\Sigma (p_0, \pa \Sigma) < C \de$, where $C$ is the constant appearing in \eqref{dist_equiv}. Let $\nu''_{\bar x}$ be the interior normal to $U''$ at $\bar x$,  and set
$$
y=\bar x +2C \de \nu''_{\bar x}
$$
(see Figure \ref{fig case 3}). We notice that the principal curvatures of $U''$ are bounded by $\mathcal{K}$ and, by choosing $\delta$ small compared to $\rho$, we have $2C \de \leq \mathcal{K}^{-1}$ and the ball $B_{2C\de}(y)$ is contained in $E_r$ and it is tangent to $U''$ at $\bar x$, with $\nu_{\bar x}''= - \bar x / |\bar x|$. Since $u(O)=\hat u (O)$ and from \cite{JEMS}[Lemma 2.5]  (where we set: $x_0=\bar x$, $c=y$ and $r=2C\de$) we find that
\begin{equation} \label{u - hat u case 3}
\|u-	\hat u\|_{C^1(B_{C\de /2}(y))} \leq C \oscH.
\end{equation}
Let $q=(y,u(y))$ and $\hat q^*=(y,\hat u(y))$ so that \eqref{u - hat u case 3} gives
$$
d(q,\hat q^*) + |\nu_q - \nu_{\hat q^*}| \leq C \oscH \,.
$$
Up to choose a smaller $\delta$, we can assume that $r=2C\delta \leq \bar r$, so that  Lemma \ref{lem_p_pstar} yields 
$$
d(q,\hat q) + |N_q - \tau_{\hat q}^q (N_{\hat q})|_q  \leq C \oscH \,,
$$
where $\hat q$ is defined as $\hat p_1$ in lemma \ref{lem_p_pstar}. Next we observe that from our construction it follows that 
$$
d_\Sigma (q,\pa \Sigma) \geq \de \,.
$$
Indeed, if we denote by $z$ the point on $\partial U_r(p_0)$ which realizes $d(q,\partial U_r(p_0))$, then
$$
d_\Sigma (q,\pa \Sigma) \geq d(q,z) = \arccosh \left(1 + \frac{|\bar q - \bar z|^2}{2q_n z_n} \right) \geq \arccosh \left(1 + \frac{2C^2 \delta^2}{q_n z_n} \right) \,.
$$
Moreover, since $|y|,|\bar z| \leq 2C\delta$, from \eqref{bounds on u} we have that $q_n \geq 1 -C_1(\rho) \delta^2$ and $z_n \geq 1 -C_1(\rho) \delta^2$ so that we can obtain $d_\Sigma (q,\pa \Sigma) \geq \delta$ by choosing $\delta$ small enough in terms of $\rho$. Being $d_\Sigma (q,\pa \Sigma) \geq \delta$ we can apply Cases 1 and 2 to conclude.

\begin{figure}[h]
\begin{tikzpicture} [scale=1.5]
\fill[gray!20!white] ({-0.51}, {0}) circle (0.8);
\fill[white]({-0.51}, {0}) circle (0.4);
\node at (-1.7,1.8) { $\pi_{\infty}$};
\node at (-1.6,0.6) { $E_r$};
\node at (0.1,1.6) {$U''$};
\node at (-0.7,0.6) {\footnotesize $2C\delta$};
\filldraw [black] (0,0) circle (0.5 pt) node[above] {\footnotesize $O$};
\filldraw [black] (0.29,0) circle (0.5 pt) node[right] {\footnotesize $\bar x$};
\filldraw [black] ({0.29-0.8}, {0}) circle (0.5 pt) node[below] {\footnotesize $ y$};
\draw [domain=100:260] plot ({2*cos(\x)}, {2*sin(\x)});
\draw (0.29,0)  --({0.29-0.8}, {0});
\draw ({2*cos(100)}, {2*sin(100)}) ..  controls(1/2,0).. ({2*cos(260)}, {2*sin(260)});
\draw({-0.51}, {0}) circle (0.8);
\draw({-0.51}, {0}) -- ({-0.51+0.8*cos(130)},{0.8*sin(130)});
\draw({-0.51}, {0}) circle (0.4);
\draw({-0.51}, {0}) -- ({-0.51+0.4*cos(200)},{0.4*sin(200)})node[left] {\footnotesize $C\delta$};;
\end{tikzpicture}
\caption{Case 3 in the proof of theorem \ref{thm approx symmetry 1 direction}.}
\label{fig case 3}
\end{figure}
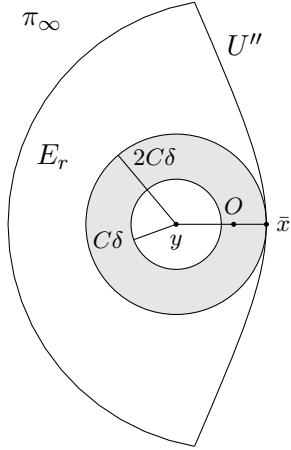

\subsubsection{Case 4: $p_0 \in \pa \Sigma$.}
This case follows from Case 3 when $d_\Sigma(p_0,\pa \Sigma) \to 0$. Indeed, in this case $E_r$ is a half-ball on $\pi_\infty$ and the argument used in Case $3$ can be easily adapted (see also the corresponding case in \cite{JEMS}).
This completes the proof of the first part of theorem \ref{thm approx symmetry 1 direction}.

\subsection{Proof of the second part of theorem \ref{main}.} Now we focus on the second part of the statement of theorem \ref{thm approx symmetry 1 direction}, showing that $\Omega$ is contained in a neighborhood of radius $C \oscH$ of $\Sigma \cup \Sigma^\pi$.

Assume by contradiction that there exists $x \in \Omega$ such that
$d(x,\Sigma \cup \Sigma^\pi)> C \oscH $. By construction, we can assume that $x\cdot e_1<0$ and hence from the connectness of $\Omega$ we can find a point $y\in \Omega$, with $y\cdot e_1<0$,  such that
$$
C \oscH < d(y,\Sigma) \leq 2 C \oscH \,.
$$
Let $p$ be a projection of $y$ over $\Sigma$. First assume that $p\cdot e_1\neq 0$. From the first part of theorem \ref{thm approx symmetry 1 direction} we have that there exists a point $\hat p \in S$ such that $\hat p=\gamma(t)$ where $\gamma$ is the geodesic satisfying $\gamma(0)=p$ and $\dot \gamma (0) = - N_p$ and such that $0\leq t \leq C \oscH$ and $|N_p-\tau_{\hat p}^p (N_{\hat p})|_p \leq C \oscH$. Moreover, we notice that by construction $\hat p$ is on the geodesic $\gamma$ connecting $y$ and $p$. Since $C\oscH$ is small (less than $\rho$ is enough), this implies that $y$ belongs to the exterior touching ball of radius $\rho$ at $p$, that is $y \not \in \Omega$, which is a contradiction. If $p \cdot e_1=0$ we obtain again a contradiction from the exterior touching ball condition since from \eqref{bellachegira} we have that $g_p(N_p,p_n e_1)\leq 1/4$. Hence the claim follows. \hspace*{\fill}  \begin{math}\Box\end{math}

\section{Proof of theorem \ref{main}} \label{section_6}
Let $\ep>0$ be the constant given by theorem \ref{thm approx symmetry 1 direction}. Let $S$ be a connected closed $C^2$-hypersurface embedded in the hyperbolic half-space $\mathbb H^n$ satisfying a touching ball condition of radius $\rho$ and such that $\oscH \leq \ep$, as in the statement of theorem \ref{main}. Given a direction $\omega$, let $\Omega_\omega$ be the maximal cap of $\Omega$ in the direction $\omega$, accordingly to the notation introduced in subsection \ref{movingplannes}. As a consequence of the second part of theorem \ref{thm approx symmetry 1 direction} we have that
\begin{equation} \label{aereo}
|\Omega_\omega|_g \geq \frac{|\Omega|_g}{2} - C \oscH \,,
\end{equation}
for some constant $C$ depending only on $n,\rho$ and $|S|_g$. Moreover the reflection $\Omega^\pi$ of $\Omega$ about $\pi$ satisfies 
\begin{equation} \label{aereo2}
|\Omega \triangle \Omega^\pi |_g = 2 (|\Omega|_g - 2 |\Omega_\omega|_g) \leq 4C \oscH \,,
\end{equation}
where $\Omega \triangle \Omega^\pi $ denotes the symmetric difference between $\Omega$ and $\Omega^\pi$. 

Now the problem consists in defining an approximate center of mass $\mathcal O$ and  quantifying the reflection about it. In the Euclidean case this step is obtained by applying the method of the moving planes in $n$ orthogonal directions and defining $\mathcal O$ as the intersection of 
the corresponding $n$ critical hyperplanes (see e.g. \cite{JEMS}). In the hyperbolic context, the situation is different since the critical hyperplanes corresponding to $n$ orthogonal directions do not necessarily intersects. However, when theorem \ref{thm approx symmetry 1 direction} is in force we can prove that they always intersect.

\begin{lemma}\label{6.1}
Let $S$ satisfy the assumptions of theorem \ref{thm approx symmetry 1 direction} and let $\{\pi_{e_1},\ldots,\pi_{e_n}\}$ be the critical hyperplanes corresponding to $\{e_1,\ldots,e_n\}$. Then 
$$
\bigcap_{i=1}^n\pi_{e_i}= \mathcal O
$$
for some $\mathcal O\in \mathbb H^n$.
\end{lemma}

\begin{proof}
It is enough to show that $\pi_{e_i}\cap \pi_{e_j}\neq \O$ for every $i,j=1,\dots n$. 
We may assume that $e_n \in S$. Let $i \neq j$. To simplify the notation we set 
$$
\pi_{k}^s= \pi_{e_k,m_{e_k}+s} \,,  \quad k \in  \{1,\ldots,n\} \,, \ s \in \mathbb{R} \,,
$$
so that the critical hyperplane in the direction $e_k$ is denoted by $\pi_k^0$. 

We prove the assertion by contradiction. Assume that $\pi_i^0 \cap \pi_j^0 =\O$ for some $i \neq j$. Then $\pi_i^0$ and $\pi_j^0$ divide $\Omega$ into three disjoint sets which we denote by $\Omega_1,\Omega_2,\Omega_3$ and we may assume that $\Omega_1$ is the maximal cap in the direction $e_i$ and $\Omega_1 \cup \Omega_2$ is the maximal cap in the direction $e_j$ (see figure \ref{O}). Moreover, in view of \eqref{aereo} we have that 
$$
|\Omega_1|_g \geq \frac{|\Omega|_g}{2} - C \oscH \,,
$$
and 
$$
|\Omega_1|_g + |\Omega_2|_g   \geq \frac{|\Omega|_g}{2} - C \oscH \,.
$$
From this, and since the reflection of $\Omega_1$ about $\pi_i^0$ is contained in $\Omega_2 \cup \Omega_3$ and the reflection of $\Omega_1 \cup \Omega_2$ about  $\pi_j^0$ is contained in $\Omega_3$, we have that 
$$
|\Omega_2|_g \leq 2 C \oscH \,.
$$
We notice that for every $k=1,\ldots,n$, we have that $\pi_{k}^{s+t}$ and $\pi_k^{s-t}$ are the two connected components of the set of points which are far $t$ from $\pi_k^s$. 
We define 
$$
\ell = \min \{d(\pi_i^0\cap \Omega,\pi_j^0\cap \Omega), \ i,j=1,\ldots,n\, \text{ and }  i \neq j \} \,.
$$
Since $\pi_i^0$ and $\pi_j^0$ do not intersect  and $S \subset \mathsf{B}_{\diam (S)}(e_n)$, we have that $\ell>0$ and proposition \ref{diametro} implies that $\ell$ depends only on $n$, $\rho$ and $|S|_g$. 
Therefore 
$$
\Omega_2 \supseteq \mathcal{E}_1 := \bigcup_{s\in(0,\ell)} \Omega \cap \pi_j^s \,,
$$
and hence $|\mathcal{E}_1|_g \leq 2 C \oscH$.
By reflecting $\mathcal{E}_1$ about $\pi_i^0$ we obtain that most of the mass of $\Omega_1$ must be at distance more than $\ell$ from $\pi_i^0$, i.e. the set $\Omega_{e_i, \ell} := \bigcup_{s\in(\ell,+ \infty)} \Omega \cap \pi_i^s$ is such that
$$
|\Omega_{e_i,\ell} |_g \geq \frac{|\Omega|_g}{2} - 2 C \oscH \,.
$$ 
Since $d(\Omega_{e_i,\ell}, \pi_j^0 \cap \Omega) \geq 2 \ell$ we have that most of the mass of $\Omega_3$ is at distance $2 \ell $ from $\pi_j^0$. This implies that  the set 
$$
\mathcal{E}_2 = \bigcup_{s\in(-2\ell,\ell)} \Omega \cap \pi_i^s 
$$
is such that $|\mathcal{E}_2|_g \leq 4 C \oscH$. By iterating this argument above we find $m \in \NN$ such that $m \ell > \diam(S)$ and
$$
0 = |\Omega_{e_i,m\ell} |_g \geq \frac{|\Omega|_g}{2} - (m+1) C \oscH \,.
$$ 
This leads to a contradiction provided that $C\oscH$ is small in terms of $n$, $\rho$ and $|S|_g$. Therefore 
$\pi_{e_i}\cap \pi_{e_j}\neq \O$. 
\end{proof}

\begin{figure}[h]
$$
\begin{tikzpicture}
\shade[rotate=60](2,1.7) ellipse (60pt and 40pt);
\filldraw [black] (0,1) circle (1pt) node[right] {\footnotesize{$e_n$}};
\node at (-2.5,4.4) {$\mathbb H^2$};
\draw[rotate=60](2,1.7) ellipse (60pt and 40pt);
\draw[->,dashed] (0,0) -- (0,5) node[right]{\footnotesize{$\gamma_{e_{2}}$}};
\draw[->] (-4,0) -- (4,0);
\draw (3.4,0) arc (0:180:3.4);
\draw[dashed] (1,0) arc (0:180:1);
\node at (1.3,0.2) {\footnotesize{$\gamma_{e_{1}}$}};
\node at (-2.4,0.2) {\footnotesize{$\pi_{e_{1}}$}};
\node at (3.8,.4) {\footnotesize{$\pi_{e_{2}}$}};
\node at (-0.7,3.7) {\footnotesize{$\Omega_1$}};
\node at (-0.7,2.5) {\footnotesize{$\Omega_2$}};
\node at (-1.2,0.9) {\footnotesize{$\Omega_3$}};
\node at (1.3,4) {\footnotesize{$S$}};
\draw (-0.2,0) arc (0:180:1.3);
\node at (3.8,-.2) {\footnotesize{$\pi_\infty$}};
\end{tikzpicture}
$$
\caption{A picture of the proof of lemma \ref{6.1} in $\mathbb H^2$. Here $e_j=e_1$ and $e_i=e_2$. }
\label{O}
\end{figure}
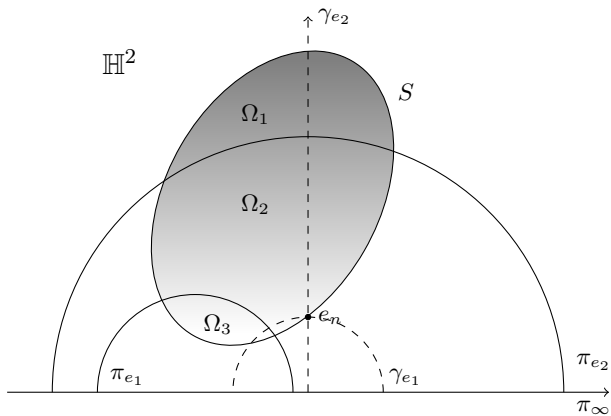

We refer to the point  $\mathcal O = \bigcap_{i=1}^{n} \pi_{e_i}$ as to the {\em the approximate center of symmetry}.
Note that,  the reflection $\mathcal{R}$ about $\mathcal O$ can be written as
$$
\mathcal{R}(p) = \pi_{e_1} \circ \cdots \circ \pi_{e_n} (p)\,,
$$
where we identify $\pi_{e_i}$ with the reflection about the corresponding hyperplane.

Next we show that if $\oscH$ is small enough, then $\pi_{\omega}$ is close to $\mathcal O$, for every direction $\omega$.  
%
%
\begin{lemma} \label{lemmarta}
There exist $\ep,C>0$ depending on $\rho,n$ and $|S|_g$ such that if the mean curvature of $S$ satisfies 
$
\oscH\leq \ep\,,
$
then 
$$
d(\mathcal O, \pi_{\omega}) \leq C \oscH\,.
$$
\end{lemma}

\begin{proof}
We may assume $\mathcal O\in \pi_{\omega,m_{\omega}-\mu}$, for some $\mu>0$  (otherwise we switch $\omega$ and  $-\omega$). Now  
we argue as in lemma 4.1 in \cite{CFMN}. We define $\mathcal{R}(\Omega)=\{ \mathcal{R}(p): \, p \in \Omega\}$. By choosing $\ep$ as the one given by theorem \ref{thm approx symmetry 1 direction}, from \eqref{aereo} and being $\mathcal{R}$ the composition of $n$ reflections, we have that
$$
|\Omega \triangle \mathcal{R}(\Omega)|_g \leq C \oscH \,,
$$
where $C$ is a constant depending on $n$, $\rho$ and $|S|_g$. It is clear that $d(\mathcal O, \pi_\omega) \leq \diam(S)$. We denote by $\Omega^{\pi_\omega}$ the reflection of $\Omega$ about $\pi_\omega$ and from \eqref{aereo2} we have that
$$
|\Omega \triangle \Omega^{\pi_\omega}|_g \leq C \oscH \,.
$$
Then the maximal cap $\Omega_{\omega}$ satisfies 
$$
|\Omega \cap \mathcal{R} (\Omega_{\omega}) |_g = |\mathcal{R}(\Omega) \cap \Omega_{\omega}|_g \geq |\Omega_{\omega}|_g - |\Omega \triangle \mathcal{R}( \Omega)|_g \geq \frac{|\Omega|_g}{2} - C \oscH \,,
$$
and from  
$$
\mathcal{R}(\Omega_{\omega}) \subset \bigcup_{s<0} \pi_{\omega,m_{\omega}-s} \,,
$$
we obtain that 
\begin{equation} \label{striscia_piccola}
\mu_0:= \big| \{ \Omega \cap \pi_{\omega,s}:\  m_\omega-\mu <s<m_{\omega}  \} \big|_g \leq C \oscH \,.
\end{equation}
Let 
$$
\mu_k= \big| \{ p \in  \Omega \cap \pi_{\omega,s}:\ m_\omega + (k-1)\mu <s<m_{\omega} + k \mu \} \big|_g
$$
for $k\in \NN$. We notice that by construction of the method of the moving planes we have that $\mu_k$ is decreasing, and hence
$$
\mu_k \leq \mu_0 \leq  C \oscH .
$$
Let $\Lambda = \sup\{s \in \mathbb{R}:\ \Omega \cap \pi_{\omega,m_\omega - \mu + s} \neq \emptyset \}$. It is clear that  
$$
\Lambda \leq \diam(\Omega) \,.
$$
Define $k_0$ as the smallest integer such that
$$
k_0 m_\omega\leq \diam(\Omega) \leq (k_0+1) m_\omega \,.
$$
From \eqref{aereo} we have
$$
\frac{|\Omega|_g}{2} - C\oscH \leq |\Omega_\omega |_g\leq \sum_{k=0}^{k_0} \mu_k  \leq k_0 \mu_0 \leq \frac{\diam(\Omega)}{m_\omega} C \oscH \,.
$$
Since $\diam (\Omega) \leq \diam(S)$, from proposition \ref{diametro} and assuming that $\oscH$ is less than a small constant depending on $n$, $\rho$ and $|S|_g$ we have that 
$$
m_\omega \leq C \oscH\,,
$$
where $C$ depends on $n$, $\rho$ and $|S|_g$.
\end{proof}

We are ready to complete the proof of theorem \ref{main}. Let $\ep$ be as in lemma \ref{lemmarta} and assume that the mean curvature of $S$ satisfies $\oscH\leq \ep\,.$
Let
$$
r=\sup\{s>0:\ \mathsf{B}_s(\mathcal O)\subset \Omega\} \quad \textmd{ and } \quad R=\inf\{s>0:\ \mathsf{B}_s(\mathcal O)\supset \Omega\} \,,
$$
so that $S \subset \overline{\mathsf{B}}_R \setminus  \mathsf{B}_r$. We aim to prove that  
$$
R-r \leq C \oscH \,,
$$ 
for some $C$ depending only on $n,\rho$ and $|S|_g$.

Let $p,q \in S$ be such that $d(p,\mathcal O)=r$ and $d(q,\mathcal O)=R$. We can assume that $p \neq q$ (otherwise the assertion is trivial). Let $t=d(p,q)$, 
$$
\omega:=\frac{1}{t}\tau_{p}^{e_n}(\exp_p^{-1}(q)) 
$$
and consider $\pi_\omega$. Let $\gamma:(-\infty,+\infty) \to \mathbb{H}^n$ be the geodesic such that $\gamma(s_p)=p$ and $\gamma(s_q)=q$.
We denote by $z$ the point on $\pi_\omega$ which realizes the distance of $\mathcal O$ from $\pi_\omega$. By construction $p\in \pi_{\omega,s_p}$ and $q\in \pi_{\omega,s_q}$ with $s_q=s_p+t$. We first prove that $d(q,z)\leq d(p,z)$. By contradiction assume that $d(q,z) > d(p,z)$. Since $q$ and $p$ belong to a geodesic orthogonal to the hyperplanes $\pi_{\omega,s}$ and $s_p< s_q$, then $s_q>m_\omega$. Since $\pi_\omega=\pi_{\omega,m_\omega}$ corresponds to the critical position on the method of moving planes in the direction $\omega$, we have that $\gamma(s) \in \Omega$ for any $s \in (m_\omega,s_q)$. Since $s_p<s_q$ we have that $|s_p-m_\omega| \geq |s_q-m_\omega|$ and being $\gamma$ orthogonal to $\pi_\omega$ we obtain $d(q,z)\leq d(p,z)$, which gives a contradiction.

From $d(q,z)\leq d(p,z)$ and by triangular inequality, we find
$$
r \geq R - d(\mathcal O, z) = R - d(\mathcal O, \pi_m) 
$$
and lemma \ref{lemmarta} implies $R-r \leq C \oscH $ and the proof is complete. \hspace*{\fill}  \begin{math}\Box\end{math}

%
\section{Proof of corollary \ref{main2}} \label{section_proof_main2}
The proof is analogous to the proof of \cite[Theorems 1.2 and 1.5]{CFMN}. We first prove an intermediate result, which proves that $S$ is a graph over $B_{r}$, and moreover it gives a first (non optimal) bound on $\|\Psi\|_{C^1(\partial B_r)}$, i.e. it gives that $\|\Psi\|_{C^1(\partial B_r)} \leq C (\oscH)^{1/2})$. Then we obtain the sharp estimate \eqref{Lipschitz_bound} by using elliptic regularity theory.

Let ${\sf B}_{r}(\mathcal O)$ and ${\sf B}_{R}(\mathcal O)$ be such that $0 \leq R - r \leq C \oscH$ and let $0<t<r- C\oscH$. For any point $p \in S$ we consider the set $\mathcal E^-(p)$ consisting of points of $\mathbb{H}^n$ belonging to some geodesic path connecting $p$ to the boundary of ${\sf B}_{t}(\mathcal O)$ tangentially. Then we denote by $\mathcal C^-(\mathcal O)$ the geodesic cone enclosed by $\mathcal E^-(p)$ and the hyperplane containing   $\mathcal E^-(p)\cap {\sf B}_{t}(\mathcal O)$. 
Moreover, we define $\mathcal C^+(p)$ as the reflection of $\mathcal C^-(p)$ with respect to $p$.

We first show that for any $p \in S$ we have that $\mathcal C^-(p)$ and $\mathcal C^+(p)$ are contained in the closure of $\Omega$ and in the complementary of $\Omega$, respectively. Moreover, the axis of $\mathcal C^-(p)$ is part of the geodesic path connecting $p$ to $\mathcal O$, and this fact will allow us to define a diffeomorphism between $S$ and $\partial  {\sf B}_{r}$.
We will prove that the interior of $\mathcal C^-(p)$ is contained in $\Omega$. An analogous argument shows that $\mathcal C^+(p)$ is contained in the complementary of $\Omega$. 

We argue by contradiction. Assume $p\notin {\sf B}_{r}(\mathcal O)$ (otherwise the claim is trivial) and that there exists a point $q \in \mathcal C^-(p) \cap \partial {\sf B}_{t}(\mathcal O)$ such that the geodesic path $\gamma$ connecting $q$ to $p$ is not contained in $\Omega$. Let $z$ be a point on $\gamma$ which does not belong to the closure of $\Omega$.  Let 
$$
\omega:=\frac{1}{d(p,q)}\tau_{e_n}^{q}(\exp_{q}^{-1}(p))
$$
and consider the critical hyperplane $\pi_{\omega}$ in the direction $\omega$. Since $z$ does not belong to the closure of $\Omega$, the method of the moving planes \lq\lq stops\rq\rq\ before reaching $z$ and therefore $z\in \pi_{\omega,s_z}$ for some $s_z\leq m_{\omega}$. Moreover, by construction $q\in \pi_{\omega,s_q}$ with $s_q\geq s_0$, where $s_0$ is such that $\mathcal O\in  \pi_{\omega,s_0}$.   Since $s_z-s_q=d(z,q)$ and $d(z,\mathcal O)\geq r$ we have
$$
d(\mathcal O, \pi_\omega) = m_{\omega}-s_0 \geq s_z -s_0\geq s_z-s_q=d(z,q)\geq d(z,\mathcal O)-d(\mathcal O,q)=d(z,\mathcal O)-t \geq r-t \,;
$$
being $0<t<r- C\oscH$ and from lemma \ref{lemmarta}, we obtain 
$$
C \oscH<r-t \leq d(\mathcal O, \pi_\omega)  \leq C \oscH \,,
$$
which gives a contradiction.

We notice that by fixing any $t=r-\ep/2$, from the argument above we have that for any $p \in S$ the geodesic path connecting $p$ to $\mathcal O$ is contained in $\Omega$. This implies that there exists a $C^2$-regular map $\Psi: \partial {\sf B}_{r}(\mathcal O) \to \mathbb{R}$ such that
$$
F(p) = \exp_x(\Psi(p)N_p) \,,
$$
defines a $C^2$-diffeomorphism  from ${\sf B}_r$ to $S$. 

Now we make a suitable choice of $t$ in order to prove that 
\begin{equation}\label{fristLipbound}
\|\Psi\|_{C^1}\leq C({\rm osc}\,H)^{1/2}\,. 
\end{equation} 
Indeed, by choosing $t=r - \sqrt{C\oscH}$ we have that for any $p \in S$ there exists a uniform cone of opening $\pi - \sqrt{C\oscH}$ with vertex at $p$ and axis on the geodesic connecting $p$ to $\mathcal O$. This implies that $\Psi$ is locally Lipschitz and the bound \eqref{fristLipbound} on $\|\Psi\|_{C^1}$ follows (see also \cite[Theorem 1.2]{CFMN}).

Finally we prove the optimal linear bound $\|\Psi\|_{C^{1,\alpha}}\leq C{\rm osc}\,H$ by using elliptic regularity. Let $\phi\colon U\to \partial{\sf B}_r$ be a local parametrization of $\partial{\sf B}_r$, $U$ being an open set of $\RR^{n-1}$. By the first part of the proof, $F\circ \phi$ gives a local parametrization of $S$. A standard computation yields that we can write 
$$
L(\Psi\circ \phi)=H(F\circ\phi)-H_{\mathsf{B}_r}
$$
where $H_{\mathsf{B}_r}$ is the mean curvature of $\partial \mathsf{B}_r$ and $L$ is an elliptic operator which, thanks to the bounds on $\Psi$ above, can be seen as a second order linear operator acting on $\Psi\circ \phi$. Then \cite[Theorem 8.32]{GT} implies the bound on the $C^{1,\alpha}$-norm of $\Psi$, as required. 

\appendix

\section{A general result on Riemannian manifolds with boundary}\label{appendix}
Let $(M,g_M)$ be a $\kappa$-dimensional orientable compact Riemannian $C^2$-manifold with boundary. For $\delta, r\in \RR^+$, $z\in M$ we denote
$$
M^\delta = \{p\in M\,:\ d_M(p,\pa M) > \delta \}\,,\quad \mathcal{B}_r(z)=\{p\in M\,\,:\,\, d_M(z,p)<r\}\,,
$$
where $d_M$ is the geodesic distance on $M$ induced by $g$.

\begin{proposition} \label{prop harnack chain}
Assume that there exist positive constants $c$ and $\delta_0$ such that
\begin{equation} \label{Dr lower bound}
|\mathcal{B}_r(z)|_{g_M} \geq cr^\kappa,
\end{equation}
and $\mathcal{B}_r(z)$ belongs to the image of the exponential map, for every $z\in M^{\delta}$ and  $0<r\leq \delta<\delta_0$. Fix $p$ and $q$ in a connected component of $M^{\delta}$. Then there exists a piecewise geodesic path $\gamma\colon  [0,1]\to M^{\delta/2}$ connecting $p$ and $q$ of length bounded by $\delta N_{\delta}$ where 
\begin{equation}\label{n 0}
N_\de:=\max\left(4,\frac{2^{\kappa} |M|_{g_M}}{c \de^{\kappa}} \right) \,.
\end{equation}
\end{proposition}

\begin{proof}
Let $\tilde \ga=\tilde \ga(t)$ be a continuous path connecting $p$ and $q$  in $M^\de$. Following the approach in \cite[Lemma 3.2]{JEMS}, we can construct a chain of {\em pairwise disjoint} geodesic balls $\{\mathcal{B}_1,\dots, \mathcal{B}_{I}\}$ of radius $\frac{\delta}{2}$ such that: $\mathcal{B}_1$ is centered at $p$; $\mathcal{B}_i$
is centered at $c_i=\tilde \gamma(t_i)$; the sequence ${t_i}$ is increasing; $\mathcal{B}_I$ contains $q$; $\mathcal{B}_i$ is tangent to $\mathcal{B}_{i+1}$ for any $i=1,\ldots,I-1$.
Since
\begin{equation*}
\Big{|} \bigcup_{i=1}^I \mathcal{B}_{i} \Big{|}_{g_M} \leq |M|_{g_M},
\end{equation*}
from \eqref{Dr lower bound} we get $I \leq N_\de$.
For every $i$ we choose a tangency point $p_i$ between 
$\mathcal{B}_i$ and $\mathcal{B}_{i+1}$. The piecewise geodesic path $\gamma$ is then constructred by connecting $c_i$ with $p_i$ and $p_i$ with $c_{i+1}$ by using geodesic radii, for $i=1,\dots I-2$, and connecting $c_{I-1}$ with $q$ by using a geodesic path contained in $\mathcal{B}_I$.
Hence
\begin{equation*}
{\rm length}(\ga) \leq I \delta \leq \de N_\de  \,,
\end{equation*}
as required. 
\end{proof}

In the next proposition we give an upper bound of the diameter of $M$ when  $\partial M=\O$.  The proof of the next proposition is analogue to the one of proposition \ref{prop harnack chain} and it is omitted. 
\begin{proposition} \label{diametro}
Assume $\partial M=\O$ and that there exist a constant $c,\delta>0$ such that
\begin{equation} \label{Dr lower bound}
|\mathcal{B}_r(z)|_{g_M} \geq cr^\kappa,
\end{equation}
for every $z\in M$ and $0<r\leq \delta$. Let $p$ and $q$ in $M$. Then there exists a piecewise geodesic path $\gamma\colon  [0,1]\to M$ connecting $p$ and $q$ of length bounded by $\delta N_{\delta}$ where 
\begin{equation}\label{n 0}
N_\de:=\max\left(4,\frac{2^{\kappa} |M|_{g_M}}{c \de^{\kappa}} \right) \,.
\end{equation}
In particular the diameter of $M$ is bounded by  $\delta N_{\delta}$. 
\end{proposition}


\end{document}